[1994/12/01]
\documentclass[10pt, reqno]{amsart}
\usepackage[english, activeacute]{babel}
\usepackage{amsmath,amsthm,amsxtra,amscd}
\usepackage[small,nohug,heads=vee]{diagrams} 
\diagramstyle[labelstyle=\scriptstyle]
\usepackage{epsfig}
\usepackage{amssymb}
\usepackage{latexsym}
\usepackage{amsfonts}
\usepackage{hyperref}
\usepackage{pdftricks}
\usepackage{pstricks}

\begin{psinputs}
\usepackage{pst-plot}
\end{psinputs}

\title{The Gardner equation and the stability of multi-kink solutions of the mKdV equation}
\author{Claudio Mu\~noz}
\address{Departamento de Ingenier\'ia Matem\'atica y CMM, Universidad de Chile, Casilla 170-3, Correo 3, Santiago \\ Chile}
\email{cmunoz@dim.uchile.cl}
\date{May, 2011}
\subjclass[2000]{Primary 35Q51, 35Q53; Secondary 37K10, 37K40}
\keywords{modified KdV equation, Gardner equation, integrability, multi-soliton, multi-kink, stability, asymptotic stability, Gardner transform}
\thanks{}

\chardef\bslash=`\\ 

\newtheorem{thm}{Theorem}[section]
\newtheorem{cor}[thm]{Corollary}
\newtheorem{lem}[thm]{Lemma}
\newtheorem{prop}[thm]{Proposition}

\newtheorem{defn}[thm]{Definition}

\theoremstyle{remark}

\numberwithin{equation}{section}

\newcommand{\R}{\mathbb{R}}

\newcommand{\T}{\mathbb{T}}

\newcommand{\la}{\lambda}

\newcommand{\al}{\alpha}
\newcommand{\ga}{\gamma}

\def\bm{\left( \begin{array}{cc}}
\def\endm{\end{array}\right)}

 \providecommand{\abs}[1]{\lvert#1 \rvert}

\newcommand{\ve}{\varepsilon}

\newcommand{\be}{\begin{equation}}
\newcommand{\ee}{\end{equation}}
\newcommand{\ba}{\begin{equation*}}
\newcommand{\ea}{\begin{equation*}}
\newcommand{\bea}{\begin{eqnarray}}
\newcommand{\eea}{\end{eqnarray}}
\newcommand{\bee}{\begin{eqnarray*}}
\newcommand{\eee}{\end{eqnarray*}}
\newcommand{\ben}{\begin{enumerate}}
\newcommand{\een}{\end{enumerate}}
\newcommand{\nonu}{\nonumber}

\newcommand{\eval}[2][\right]{\relax
  \ifx#1\right\relax \left.\fi#2#1\rvert}

\let\abs=\envert

\begin{document}
\begin{abstract}
Multi-kink solutions of the defocusing, modified Korteweg-de Vries equation (mKdV) found by Grosse \cite{Gr1,Gr2} are shown to be globally $H^1$-stable, and asymptotically stable. Stability in the one-kink case was previously established by Zhidkov \cite{Z} and Merle-Vega \cite{MV}. The proof uses transformations linking the mKdV equation with focusing, Gardner-like equations, where stability and asymptotic stability in the energy space are known.  We generalize our results by considering the existence, uniqueness and the dynamics of generalized multi-kinks of defocusing, non-integrable gKdV equations, showing the inelastic character of the kink-kink collision in some regimes.
\end{abstract}
\maketitle \markboth{$H^1$-stability of multi-kinks} {Claudio Mu\~noz}
\renewcommand{\sectionmark}[1]{}

\section{Introduction and main results}

\medskip

In this paper we continue our work on stability of multi-soliton solutions for some well-known, dispersive equations, started in a joint work with M.A. Alejo and L. Vega \cite{AMV}. In this opportunity, we consider the nonlinear $H^1$-stability, and asymptotic stability, of the \emph{multi-kink} solution of the \emph{defocusing}, modified Korteweg-de Vries (KdV) equation
\be\label{mKdV}
u_{t}  +  (u_{xx} - u^3)_x =0.
\ee
Here $u=u(t,x)$ is a real valued function, and $(t,x)\in \R^2$. Solutions $u=u(t,x)$ of (\ref{mKdV}) are invariant under space and time translations, and under suitable scaling properties. Indeed, for any $t_0, x_0\in \R$, and $c>0$, both $u(t-t_0, x-x_0)$ and $c^{1/2} u(c^{3/2} t, c^{1/2} x)$ are solutions of (\ref{mKdV}).  Finally,  $u(-t,-x)$ and $-u(t,x)$ are also solutions.

\medskip

From a mathematical point of view, equation (\ref{mKdV}) is an \emph{integrable model} \cite{AS}, with a Lax pair structure and infinitely many conservation laws. Moreover, equation (\ref{mKdV}) has non-localized \emph{solitons} solutions, called \emph{kinks}, namely solutions of the form
\be\label{Sol}
u(t,x) = \varphi_c (x + ct + x_0), \quad \varphi_c(s) := \sqrt{c}\ \varphi (\sqrt{c} s), \quad c>0,\; x_0\in \R,
\ee
and
\be\label{Q}
\varphi (s):= \tanh (\frac{s}{\sqrt{2}}),
\ee
which solves
\be\label{ecvarfi}
\varphi'' + \varphi - \varphi^3 =0 , \; \hbox{ in } \R, \quad \varphi(\pm \infty) =\pm 1, \quad \varphi'>0.
\ee

\medskip

On the one hand, since the Cauchy problem associated to (\ref{mKdV}) is locally well posed in $ \varphi_c(\cdot +ct) + H^1(\R)$ (cf. Merle-Vega \cite[Prop. 3.1]{MV}), each solution is indeed global in time thanks to the conservation of Energy:
\be\label{E}
E[u](t) := \frac 12 \int_\R u_x^2(t,x) dx + \frac 14 \int_\R (u^2-c)^2(t,x) dx = E[u](0).
\ee
A simple inspection reveals that this is a non-negative quantity. 

\medskip

On the other hand, the standard Cauchy problem for initial data in the Sobolev space $H^s(\R)$ is locally well-posed for $s \geq \frac 14$ (Kenig-Ponce-Vega \cite{KPV}), and globally well-posed for $s>\frac 14$ (Colliander et al. \cite{CKSTT}). This result is almost sharp since for  $s<\frac 14$ the solution map has been shown  to not be uniformly continuous, see Christ-Colliander-Tao \cite{CCT} (see also Kenig-Ponce-Vega \cite{KPV2} for an early result in the focusing case).

\medskip

It is also important to stress that (\ref{mKdV}) has in addition another less regular conserved quantity, called mass:
\be\label{M}
M[u](t) := \frac 12 \int_\R (c-u^2(t,x))dx = M[u](0).  
\ee
Of course this quantity is well-defined for solutions $u(t)$ such that $(u^2(t)-c)$ has enough decay at infinity. In particular, one has $M[\varphi_c] <+\infty.$

\medskip

Now we focus on the study of suitable perturbations of kinks solutions of the form (\ref{Sol}). This question leads to the introduction of the concepts of \emph{orbital and asymptotic stability}. In particular, since the energy (\ref{E}) is a conserved quantity --in other words, it is a Lyapunov functional--, well defined for solutions at the $H^1$-level, it is natural to expect that kinks are (orbitally) stable under small perturbations in the energy space. Indeed, $H^1$-stability of mKdV kinks has been considered initially by Zhidkov \cite{Z}, see also Merle-Vega \cite{MV} for a complete proof, including an adapted well-posedness theory. We recall that their proof is strongly based in the {\bf non-negative} character of the energy (\ref{E}) around a kink solution $\varphi_c$, which balances the bad behavior of the mass (\ref{M}) under general $H^1$-perturbations of a kink solution.

\medskip

For additional purposes, to be explained later, we recall that in \cite{MV}, the main objective of Merle and Vega was to prove that solitons of the KdV equation
\be\label{KdV}
u_t + (u_{xx} + u^2)_x =0,
\ee
were $L^2$-stable, by using the \emph{Miura transform} 
\be\label{Mi}
M[v](t,x) :=  \frac 32 c +\big[\frac 3{\sqrt{2}} v_x - \frac 32 v^2\big](t, x-3ct).
\ee
This nonlinear $H^1-L^2$ transformation links solutions of (\ref{mKdV}) with solutions of the KdV equation (\ref{KdV}). In particular, the image of the family of kink solutions (\ref{Sol}) under the transformation (\ref{Mi}) is the well-known soliton of KdV, with scaling $2c$ (cf. \cite{MV}):
$$
M[\varphi_c (x+ct +x_0)] = Q_{2c}(x-2ct +x_0).
$$
Therefore, by proving the $H^1$-stability of single kinks --a question previously considered by Zhidkov \cite{Z}--, and (\ref{Mi}), they obtained a form of $L^2$-stability for the KdV soliton. Additionally, a simple form of asymptotic stability for the kink solution was proved. Related asymptotic results for soliton-like solutions can also be found e.g. in \cite{Cuc1,SW,PW, KK,MMarma,MMnon}.

\medskip

Kinks are also present in other nonlinear models, such as the sine-Gordon (SG) equation, the $\phi^4$-model, and the Gross-Pitaevskii (GP) equation \cite{AC,DP}. In each case, it has been proved that their are stable for small perturbations in a suitable space, cf. \cite{HPW,GSS,Z,GZ,BGSS}. Let us also recall that the SG and GP equations are integrable models in one dimension \cite{AC,DP}.

\medskip

Let us come back to the equation (\ref{mKdV}). In addition to the previously mentioned kink solution (\ref{Sol}), mKdV has \emph{multi-kink} solutions, as a consequence of the integrability property, and the Inverse Scattering method. This result, due to Grosse \cite{Gr1,Gr2}, can be obtained by a different approach using the Miura transform (\ref{Mi}), see Gesztesy-Schweinger-Simon \cite{GS,GSSi}, or the monograph by Thaller \cite{Dirac}. According to Gesztesy-Schweinger-Simon \cite{GSSi}, there are at least two different forms of multi-kink solutions for (\ref{mKdV}), which we describe below (cf. Definitions \ref{MKe} and \ref{MKo}). Moreover, they proved that the Miura transform sends these solutions towards a well defined family of multi-soliton solutions of the KdV equation, provided a \emph{criticality} property is satisfied (see \cite{GSS,GSSi} for such an assumption). 

\medskip

In order to present multi-kinks from a different point of view, we need some preliminaries.

\medskip

In \cite{AMV} (see e.g. \cite{Mu2} for a short review), Alejo-Mu\~noz-Vega showed the $L^2$-stability of KdV multi-solitons following the Merle-Vega approach and the Gesztesy-Schweinger-Simon property \cite{GSSi}, above described. However, $H^1$-stability of multi-kinks was not known at that moment. Even worse, according to our knowledge, there was no result involving stability of several kink solutions, for any type of dispersive equation with stable kinks. Instead, we avoided this problem and followed a different approach, based in the use of the {\bf Gardner transform}
\be\label{GTr}
M_\beta[v] := v - \frac 32\sqrt{2\beta} v_x - \frac 32\beta v^2.
\ee
This nonlinear map links solutions of the KdV equation and the \emph{Gardner equation} \cite{Ga0,Ga},
\be\label{Ga}
v_t + (v_{yy} + v^2 -\beta v^3 )_y=0, \; \hbox{ in } \;  \R_t \times \R_y , \quad \beta>0.
\ee
In particular, the Gardner transform sends Gardner solitons towards KdV solitons (see \cite{AMV} for further details).

However, we have realized that the existence, uniqueness and stability of multi-kinks is closely related to the solitons of the Gardner equation, and more generally, \emph{dynamical properties of defocusing gKdV equations are closely related to those of suitable focusing counterparts}. In particular, as a consequence of our results, we provide the first proof of stability for multi-kinks solutions of the mKdV equation. This result can be also considered a first step towards the understanding of the dynamics of several SG and $\phi^4$ kinks. 

\medskip

In order to explain in more detail this relationship, let us recall that the Gardner equation is also an integrable model \cite{Ga}, with soliton solutions of the form 
$$
v(t,y) := Q_{c,\beta} (y-ct),
$$
and\footnote{See e.g. \cite{Mu1,AMV} and references therein for a more detailed description of solitons and integrability for the Gardner equation.} 
\be\label{SolG}
Q_{c,\beta} (s) := \frac{3c}{1+ \rho \cosh(\sqrt{c}s)}, \quad \hbox{ with }\quad  \rho := (1-\frac 92 \beta c)^{1/2}, \quad 0<c < \frac{2}{9\beta}.
\ee
In particular, in the formal limit $\beta\to 0$, we recover the standard KdV soliton. On the other hand, the Cauchy problem associated to (\ref{Ga}) is globally well-posed under initial data in the energy class $H^1(\R)$ (cf. \cite{KPV}), thanks to the mass and energy conservation laws (see \cite{AMV} for more details).

\medskip

The first, striking connection is well-known in the mathematical physics literature, and it was in part used in the recent paper \cite{AMV}. Indeed, let $v=v(t,y)\in C(\R,H^1(\R))$ be a solution of (\ref{Ga}). Then  
\be\label{CG}
u(t,x):= b -  \sqrt{\beta} \, v \big(t,x+\frac{t}{3\beta } \big),  \quad b:=\frac{1}{3\sqrt{\beta}},
\ee
solves the mKdV equation (\ref{mKdV}).\footnote{Note that \be\label{111}u(t,x):= - b +  \sqrt{\beta} \, v \big(t,x+\frac{t}{3\beta } \big),  \quad b=\frac1{3\sqrt{\beta}},\;  \beta>0,\ee is also a solution of (\ref{mKdV}).} In terms of the Miura and Gardner transform, it reads as follows

\begin{figure}[h!]
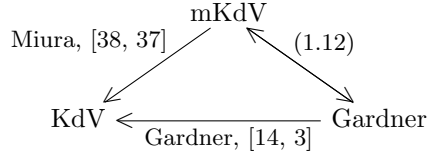

\begin{diagram}
& & \hbox{mKdV} & & \\
& \ldTo^{\small\hbox{Miura, \cite{Ga0,MV}}} & &  \luTo\rdTo^{\smaller\hbox{(\ref{CG})}} &  \\
\hbox{KdV} & & \lTo_{\small\hbox{Gardner, \cite{Ga,AMV}}} & & \hbox{Gardner} 
\end{diagram}
\caption{Transformation (\ref{CG}) in terms of Miura (\ref{Mi}) and Gardner transforms.}
\end{figure}

Note that, for $t$ fixed, (\ref{CG}) is a diffeomorphism which {\bf preserves regularity}, a key difference with respect to the Miura and Gardner transforms. Note in addition that $u$ in (\ref{CG}) is an $L^\infty$-function with nonzero limits at infinity. This will allow to consider the first class of multi-kink solutions of (\ref{mKdV}), characterized by the same positive limit $(=b)$ as $x\to \pm \infty$. Moreover, since $-u(t,x)$ is also a solution of (\ref{mKdV}), one can easily construct a solution with negative limits at infinity. This analysis motivates the following alternative approach for the multi-kink solution:

\begin{defn}[\emph{Even} multi-kink solutions, see also \cite{Gr1,Gr2,GSSi,Dirac}]\label{MKe}~

Let $\beta>0$, scaling parameters $0<c_1^0<c_2^0<\ldots <c_N^0<\frac 2{9\beta}$ and $x_1^-, \ldots, x_N^- \in \R$ be fixed numbers. We say that a solution $U_{e} (t) := U_{e} (t; c_1^0,\ldots ,c_N^0; x_1^-, \ldots, x_N^-)$ of (\ref{mKdV}) is an \emph{even multi-kink} if it satisfies
\be\label{lim1}
\lim_{t\to - \infty}  \big\| U_{e} (t)  - b + \sqrt{\beta}\sum_{j=1}^N Q_{c_j^0,\beta} (\cdot + \tilde c_j t + x_j^{-} ) \big\|_{H^1(\R)} =0,
\ee
\be\label{lim2}
\lim_{t\to + \infty}  \big\| U_{e} (t)  - b + \sqrt{\beta}\sum_{j=1}^N Q_{c_j^0,\beta} (\cdot + \tilde c_j t + x_j^{+} ) \big\|_{H^1(\R)} =0,
\ee
with $b=\frac 1{3\sqrt{\beta}}$, $\tilde c_j:= \frac 1{3\beta}- c_j^0 >0$, $x_j^{+}\in \R$ depending only on $(x_k^{-})$ and $(c_k^0)$, and $Q_{c,\beta}$ being solitons of the Gardner equation (\ref{Ga}). 
\end{defn}

\noindent
{\bf Remarks. }

\noindent
1. Let us emphasize that $\tilde c_{N}<\tilde c_{N-1}<\ldots <\tilde c_1 $, which means that bigger Gardner solitons are actually slower than the smaller ones. Note also that they move from the right to the left, as time evolves. In conclusion, as time goes to $\pm \infty$, the Gardner components of the multi-kink solution are ordered in the {\bf inverse} sense compared with the usual solitons of the Gardner equation, or any \emph{focusing} gKdV equation.

\medskip

\noindent
2. The denomination {\bf multi-kink} above comes from the fact the these solutions can be seen asymptotically as the \emph{sum} of several kinks $\pm \varphi_c$ of the form (\ref{Sol}). For instance, with our notation, given $\beta>0$ and $0<c<\frac 2{9\beta}$, an expression for the 2-kink solution is given by \cite[p. 505]{GSSi} (see also \cite[p. 273]{Dirac}) 
\be\label{2K}
U_{e}(t,x) =  b  - [\varphi_{c/2} (x+ \tilde c t + 2x_0 ) -\varphi_{c/2} (x+\tilde c t)], 
\ee
with
$$
b:=\frac{1}{3\sqrt{\beta}} ,\quad  \tilde c:=    \frac 1{3\beta} - c, \quad x_0:= \frac{1}{2\sqrt{c}} \log \Big(\frac{\sqrt{2} + 3\sqrt{\beta c}}{\sqrt{2}- 3\sqrt{\beta c}} \Big)>0,
$$
and $\varphi_c$ as in (\ref{Sol}). Note that both kinks $\pm \varphi_{c/2}$ have the {\bf same velocity} $\tilde c$, a key difference with the SG and $\phi^4$ models. After a quick computation, using the identity
$$
\tanh (a+k) -\tanh a = 2\tanh (k) \Big[ 1+\frac{\cosh(2a+k)}{\cosh(k)} \Big]^{-1},
$$
one can see that (\ref{2K}) can be written, as in (\ref{lim1})-(\ref{lim2}), using the Gardner soliton (\ref{SolG}):
$$
U_{e}(t,x) = b -\sqrt{\beta}Q_{c,\beta}(x +  \tilde c t  + x_0),
$$
which will be helpful for our purposes. From this fact one can say that in general, the function $U_{e}$ represents a {\bf $2N$-kink solution}. In terms of our point of view, it will represent $N$ different Gardner solitons attached to the non-zero constant $b$.

\medskip

The existence of a solution $U_{e}$ satisfying (\ref{lim1}) is a simple consequence of (\ref{CG}) and the behavior of the $N$-soliton solution of the Gardner equation (see also \cite{Gr1,Gr2,GSSi} and \cite[pp. 272-273]{Dirac} for the standard deduction). Indeed, from the integrable character of this last equation, given parameters $\beta>0$, $0<c_1^0<\ldots<c_N^0<\frac 2{9\beta}$, and $x_1^0,\ldots, x_N^0\in \R$, it is well-known that there exists a $N$-soliton solution of the form (see e.g. Maddocks-Sachs \cite{MS} for a similar structure)
\be\label{VN}
V^{(N)}(t,x) := V^{(N)}(x; c_j^0; x_j^0-c_j^0 t )
\ee 
of (\ref{Ga}), and which satisfies
\bea
\lim_{t\to - \infty}\big\| V^{(N)}(t) - \sum_{j=1}^N Q_{c_j^0,\beta} (\cdot- c_j^0 t -x_j^-) \big\|_{H^1(\R)} =0, \label{Minus}\\
\lim_{t\to + \infty}\big\| V^{(N)}(t) - \sum_{j=1}^N Q_{c_j^0,\beta} (\cdot- c_j^0 t -x_j^+) \big\|_{H^1(\R)} =0,\label{Plus}
\eea
for some $x_j^\pm \in \R$, uniquely depending on the set of parameters $(c_k^0,x_k^0).$ 
Moreover, note that $V^{(N)}(t)$ is {\bf unique} in the sense described by Martel in \cite{Martel1}:

\medskip

\noindent
{\bf Asymptotic uniqueness}: ~
Given $\beta>0$, $0<c_1^0<\ldots<c_N^0<\frac 2{9\beta}$, and $x_1^-,\ldots, x_N^-\in \R$, the corresponding multi-soliton $V^{(N)}$ given in (\ref{VN}) is the unique $C(\R, H^1(\R))$-solution of (\ref{Ga}) satisfying (\ref{Minus}). 

\medskip

From (\ref{CG}) we can define
\be\label{Unbe}
U_{e}(t) := b + \sqrt{\beta} V^{(N)}(t, \cdot +\frac t{3\beta}).
\ee
Note that we {\bf do not need} the criticality assumption required in \cite{GSSi,Dirac}. Following the notation of Maddocks and Sachs \cite{MS} and (\ref{VN}), we may think $U_{e}$ as a function of three independent set of variables: 
$$
U_{e}(t,x) := U_{e}(x; c_j^0, x_j^0+\tilde c_j t ), 
$$
with $\tilde c_j$ given in Definition \ref{MKe}.  Therefore, as a conclusion of the preceding analysis, and using (\ref{CG}), we get the uniqueness of the corresponding solution $U_{e}$. 

\begin{thm}[Uniqueness of even multi-kink solutions]\label{UniE}~

Let $\beta>0$, $0<c_1^0<c_2^0<\ldots <c_N^0< \frac2{9\beta}$ and $x_1^-, \ldots, x_N^- \in \R$ be fixed numbers. Then the associated even multi-kink $U_{e}$ defined in (\ref{Unbe}) is the unique solution of (\ref{mKdV}) satisfying (\ref{lim1}). 
\end{thm}

\begin{proof}
See Section \ref{2}.
\end{proof}

The second problem that we want to consider is the {\bf stability} of the multi-kink $U_{e}$. First of all, we recall some important literature.
 
\medskip

In \cite{MMT, MMan}, Martel, Merle and Tsai have showed the stability and asymptotic stability of the \emph{sum of $N$ solitons} of some generalized KdV equations, \emph{well decoupled} at the initial time, in the energy space $H^1(\R)$. We say that such an initial data is {\bf well-prepared}. Their approach is based on the construction of $N$ \emph{almost conserved quantities}, related to the mass of each solitary wave, plus the total energy of the solution. Although the proof for general nonlinearities is not present in the literature, it is a direct consequence of \cite{MMT} (see also Section 5 in \cite{MMan}.). An important remark to stress is that  their proof applies even for non-integrable cases, provided they have {\bf stable solitons}, in the sense of Weinstein \cite{We}. In the particular case of the Gardner equation, this condition reads  
\be\label{WC}
\partial_c \int_\R Q_{c,\beta}^2(s)ds >0, \quad \hbox{ for $c<\frac 2{9\beta}$.}
\ee
 This inequality is directly verifiable in the case of Gardner solitons, see (\ref{dQc}). From this result and Definition \ref{MKe} we claim the following

\begin{thm}[Stability of \emph{even} multi-kink solutions]\label{MT1}~

The family of multi-kink solutions $U_{e} (t)$ from Definition \ref{MKe} and (\ref{Unbe}) is global-in-time $H^1$-stable, and asymptotically stable as $t\to \pm\infty$.

\end{thm}

In Section \ref{2} we give a precise, $\ve$-$\delta$ formulation of this result. See Theorem \ref{MT1b}.

\bigskip

There is a second type of multi-kink solutions for (\ref{mKdV}), which is actually the best known one. Here, the standard kink $\varphi_c$ in (\ref{Sol}) and the Gardner equation play once again a crucial and surprising r\^ole. Indeed, let $\beta>0$ be a fixed parameter and suppose that one has a solution of (\ref{mKdV}) of the form (the reader may compare with (\ref{CG}) and (\ref{111}))
\be\label{CG2}
u(t,x) := \varphi_{c}(x+c t)  + \sqrt{\beta} \tilde u(t,x+\frac t{3\beta} ),  \quad c:= \frac{1}{9\beta},
\ee
and $\tilde u(t)\in H^1(\R)$. Then $\tilde u(t,y)$ satisfies the equation
\be\label{eqtu}
\tilde u_t + (\tilde u_{yy} + \tilde u^2 -\beta \tilde u^3)_y = 3((\varphi_{c}^2-c)\tilde u  +\sqrt{\beta} (\varphi_{c} +\sqrt{c})\tilde u^2)_y,
\ee
with $\varphi_c =\varphi_c (y- 2ct)$. In particular, if the support of $\tilde u(t)$ is mainly localized in the region where $\varphi_c \sim -\sqrt{c}$, namely $y\ll 2ct $, then the right hand side above is a {\bf small perturbation} of the left hand side, a Gardner equation with parameter $\beta>0$. As an admissible function $\tilde u$, we can take e.g. a sum of Gardner solitons:
$$
\tilde u(t,y) \sim \sum_{j=1}^{N-1} Q_{c_j,\beta} (y-c_j t ), \qquad 0<c_1<c_2\ldots <c_{N-1}<\frac 2{9\beta} = 2c,
$$
with support  localized in the region $c_{1} t \lesssim y \lesssim c_{N-1} t$, for $t\gg 1$. In particular, one has $c_{N-1}t \ll 2ct $ for $t\gg 1$, which is a necessary condition for the existence of a solution of the form (\ref{CG2}). Note in addition that Figure \ref{F2} can be adapted in the following way:

\begin{figure}[h!]
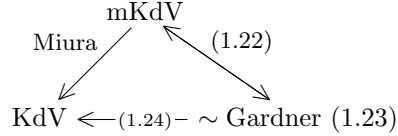
\label{F2}
\begin{diagram}
& & \hbox{mKdV} & & \\
& \ldTo^{\small\hbox{Miura}} & &  \luTo\rdTo^{\small\hbox{(\ref{CG2})}} &  \\
\hbox{KdV} & & \lTo~{(\ref{CG4})} & &\sim \hbox{Gardner } (\ref{eqtu}) 
\end{diagram}
\caption{The generalized diffeomorphism (\ref{CG2}) linking (\ref{eqtu}) and mKdV. }
\end{figure}
The transformation linking (\ref{eqtu}) and KdV is nothing but
\be\label{CG4}
\tilde u(t,x) \mapsto Q_{2c}(x-2ct) -3\sqrt{\beta} \varphi_c (x-2ct) \tilde u(t,x)  +\frac 32 \sqrt{2\beta} \tilde u_x(t,x) -\frac 32\beta \tilde u^2(t,x),
\ee
(compare with the Gardner transform (\ref{GTr})).

Finally, the same argument can be done in the case of a solution of the form $u(t,x) := \varphi_{c}(x+c t)  - \sqrt{\beta} \hat u(t,x+\frac t{3\beta} ),$ and the equation for $\hat u (t,y)$,
$$
\hat u_t + (\hat u_{yy} + \hat u^2 -\beta \hat u^3)_y = 3((\varphi_{c}^2-c)\hat u  +\sqrt{\beta}(\sqrt{c} -\varphi_{c})\hat u^2)_y,
$$
provided $\hat u$ is supported mainly in the region $\{ \varphi_c \sim \sqrt{c}\}$.  These two new ideas allow us to consider the following definition of a multi-kink solution, from the point of view of the Gardner equation:

\begin{defn}[\emph{Odd} multi-kink solutions, \cite{GSSi,Dirac}]\label{MKo}~

Let $N\geq 2$, $\beta>0$, scaling parameters $0<c_1^0<c_2^0<\ldots <c_{N-1}^0<\frac 2{9\beta}$ and $x_1^0, \ldots, x_{N}^0 \in \R$ be fixed numbers. We say that a solution $U_{o} (t) := U_{o} (t; c_j^0; x_j^0)$ of (\ref{mKdV}) is an \emph{odd multi-kink} solution if it satisfies
\be\label{lim3}
\lim_{t\to - \infty}  \big\| U_{o} (t) - \varphi_{c_N^0}(\cdot + c_N^0 t + x_N^{-}) + \sqrt{\beta} \sum_{j=1}^{N-1} Q_{c_j^0,\beta} (\cdot + \tilde c_j t + x_j^{-} ) \big\|_{H^1(\R)} =0,
\ee
\be\label{lim4}
\lim_{t\to + \infty}  \big\| U_{o} (t) - \varphi_{c_N^0}(\cdot + c_N^0 t + x_N^{+}) - \sqrt{\beta} \sum_{j=1}^{N-1} Q_{c_j^0,\beta} (\cdot + \tilde c_j t + x_j^{+} ) \big\|_{H^1(\R)} =0,
\ee
with $c_{N}^0:= \frac 1{9\beta}$, $\tilde c_j:=  \frac 1{3\beta}-c_j^0 >0$ and $x_j^{\pm}\in \R$ depending only on $(c_k^0)$. Finally,  $ \varphi_{c}$ is a kink solution (\ref{Sol}) with scaling $c$, and $Q_{c,\beta}$ is a soliton solution of the Gardner equation (\ref{Ga}).
\end{defn}

One can also say that $ U_{o}$ is composed by $(2N-1)$ single kinks, in other words, it is a {\bf $(2N-1)$-kink solution}. Additionally, as above mentioned, one may think this solution as function composed of three different class of parameters:
$$
U_{o} (t,x)  =U_{o} (x; c_j^0; \tilde c_j t + x_j^0), \quad (\tilde c_N := c_N^0).
$$
From the point of view of the Gardner equation, this solution represents a big kink, solution of mKdV, with attached $(\pm)$Gardner solitons ordered according their corresponding velocities $\tilde c_j$. Note finally that solitons move from the right to the left.  
\begin{figure}[h!]
  \medskip
  \begin{flushleft}
  \includegraphics[angle=180,scale=0.450]{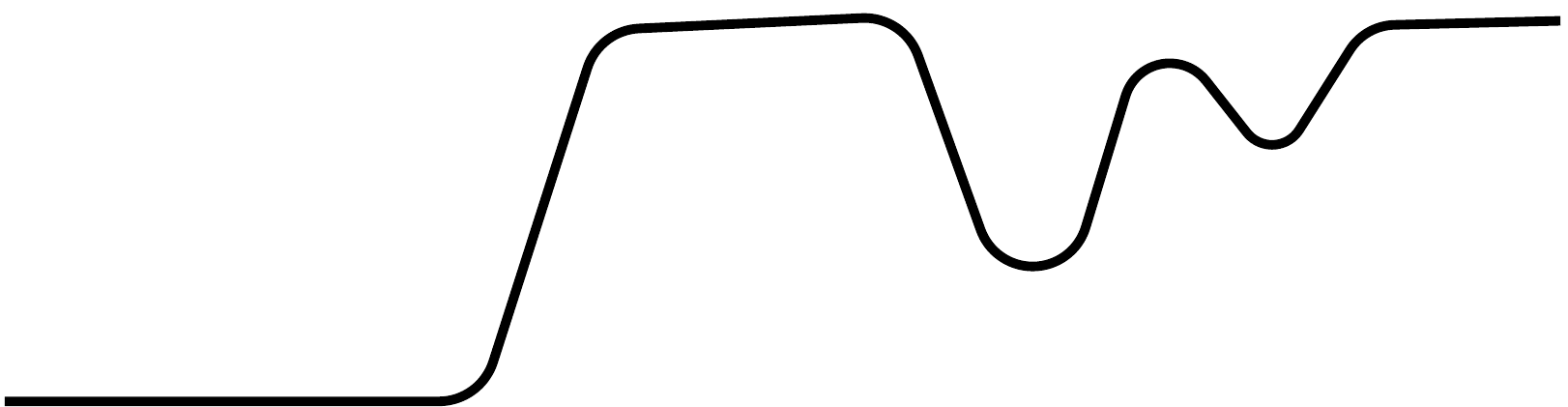} 
  \end{flushleft}
  \bigskip
\begin{flushright}
  \includegraphics[scale=0.450]{Diagrama1.pdf}
  \end{flushright}
    \caption{A schematic design of the evolution in time of a 5-kink solution of mKdV, composed of a big kink $\varphi_{c_3^0}$ and two Gardner solitons, $Q_{c_1^0,\beta}$ and  $Q_{c_2^0,\beta}$, $0<c_1^0<c_2^0<2c_3^0$, and $\beta = \frac 1{9c_3^0}$. Below, the behavior as $t\to -\infty$; above, the behavior as $t\to +\infty$. Each part is ordered according to their respective velocity $-\tilde c_j:= c_j^0- 3c_3^0<0$, $j=1,2$. Note that $-\tilde c_1< -\tilde c_2< - c_3^0$, which means that the smallest soliton $Q_{c_1^0,\beta}$ is actually the fastest one.}
\end{figure}

\medskip

The proof of the existence of this family is not direct, although it can be explicitly obtained from the solutions found by Grosse \cite{Gr1,Gr2}, using the Inverse Scattering method (see also \cite{GSSi}, or \cite[pp. 270-272]{Dirac} for an alternative procedure involving the inverse Miura transform). In this paper we present a third proof, which gives in addition a uniqueness property, uniform estimates and does not require the criticality property considered in \cite{GSSi,Dirac}. The uniqueness is, of course, modulo the $2N$-parameter family $(c_j^0, x_j^-)$.

\begin{thm}[Existence and uniqueness of odd multi-kink solutions]\label{EoMK}~

Let $N\geq 2$, $\beta>0$, $c_N^0= \frac 1{9\beta}$,  $0<c_1^0<c_2^0<\ldots <c_{N-1}^0<\frac 2{9\beta}$ and $x_1^-, \ldots, x_{N}^- \in \R$ be fixed numbers. There exists a unique solution $U_{o} (t) $ of (\ref{mKdV}) satisfying (\ref{lim3}).
\end{thm}

\begin{proof}
See Section \ref{4}. 
\end{proof}

The next result is a positive answer to the open question of stability of odd multi-kinks.

\begin{thm}[Stability of odd multi-kink solutions]\label{MT2}~

The family of multi-kink solutions $U_{o} (t)$ from Definition \ref{MKo} is global-in-time $H^1$-stable.
\end{thm}

We prove this result in Section \ref{3}. In particular, we give a more precise statement in Theorem \ref{MT2b}.

\subsection*{General nonlinearities} We point out that our results, starting from transformations (\ref{CG})-(\ref{CG2}), and the Zhidkov theory developed in \cite{Z}, can be made even more general and include a wide range of non-integrable, defocusing gKdV equations. 

\medskip

Indeed, in the next paragraphs we first introduce the notion of {\bf generalized, even and odd multi-kink solutions}. Of course these objects have to match with those considered in Definitions \ref{MKe} and \ref{MKo}, for the special case of the integrable mKdV model. Second, we study the existence, uniqueness and stability of these new solutions in the case of \emph{well-prepared} initial data. Next, we consider some particular collision problems, in the spirit of \cite{MMcol1,MMcol2,Mu1} (note that the collision problem makes sense since we consider non-integrable equations).  The method used is the same as in the previous results, so \emph{we will skip most of the proofs}. We emphasize that the main idea is to exploit the properties contained in the following figure: 
\begin{figure}[h!]
\begin{diagram}
 u\in \hbox{ defocusing gKdV (even)} &  \lTo & {\smaller\hbox{$(u\sim b+ \tilde u)$}} &  \rTo & \tilde u \in \hbox{focusing gKdV} \\
 u\in \hbox{ defocusing gKdV (odd)} &  \lTo & {\smaller\hbox{$(u\sim\varphi_c + \tilde u)$}} &  \rTo & \tilde u \in \hbox{focusing gKdV}
\end{diagram}
\caption{The generalized transformations (\ref{CG})-(\ref{CG2}) linking a defocusing gKdV with a focusing gKdV equation.}\label{Dia}
\end{figure}

\noindent
Let us consider the \emph{generalized, defocusing} KdV equation
\be\label{dgKdV}
u_t + (u_{xx} -f(u))_x=0,\quad u =u(t,x)\in \R.
\ee 
Here $f:\R\to \R$ is a non-linear term, with enough regularity, to be specified below. From now on, we will assume the following hypotheses:

\medskip

\ben
\item[{\bf(a)}] There exists $b_0\in \R$ such that $f^{(k)}(b_0) \neq 0$, for some $k\in \{ 2,3,4\}$. 

\medskip

\noindent
Let $k_0\in\{2,3,4\}$ be the first integer $k$ satisfying this property. We assume $f$ is of class $C^{k_0+1}(\R)$.

\medskip

\item[\bf (b)] If $f^{(2)}(b_0)=0$ and $f^{(3)}(b_0)\neq 0$, then $f^{(3)}(b_0)< 0$.
\een

\medskip

The reader may compare e.g. with the integrable case $f(s)=s^3$, where for $b_0\neq 0$ one has $f''(b_0) = 6b_0\neq 0$. Another important example is the \emph{cubic-quintic} nonlinearity $f(s) := s^3 + \mu s^5$, $\mu \in \R-\{0\} $ fixed, for which 
$$
f''(b_0) = 2b_0 (3 + 10\mu b_0^2)\neq 0,  
$$
provided $b_0\neq 0$ and $b_0^2\neq -\frac{3}{10 \mu}$ in the case $\mu<0$. Otherwise, we are in a degenerate case and $f''(b_0) =0$. Computing the third derivative, one has
$$
f^{(3)} (b_0) = 6 (1+10 \mu b_0^2).
$$
For any $\mu \neq 0$ one has that $b_0=0$ does not satisfy {\bf (b)} above and therefore is not allowed. In the case $\mu<0$, $b_0^2= -\frac{3}{10 \mu}$, one has  $f^{(3)} (b_0)  = -12<0,$ which is an admissible case.

\medskip

Under these two assumptions and using \cite[Remark 2]{Martel1}, we prove existence and uniqueness of generalized even multi-kinks for (\ref{dgKdV}), satisfying the equivalent of (\ref{lim2}) (cf. Definition \ref{MKe}).

\begin{thm}[Existence and uniqueness of generalized, even multi-kinks]\label{Geven}~

Let $N\geq 2$ and $b_0 \in \R$ be such that ${\bf (a)}$-${\bf (b)}$ above are satisfied. There exists $c^*:=c^*(f,b_0)>0$ such that, for all $0<c_1^0<c_2^0<\ldots <c_{N}^0<c^*$, the following holds. There exists a unique solution $U_e=U_e(t) $ of (\ref{dgKdV}) such that $U_e(t) -b_0 \in C(\R, H^1(\R))$, and it satisfies
\be\label{Pure}
\lim_{t\to +\infty} \big\|U_e(t) -b_0 -b_1\sum_{j=1}^N Q_{c_j^0}(\cdot + \tilde c_j t) \big\|_{H^1(\R)} =0,
\ee
with\footnote{In the case $k_0 =2$ both $b_1$ and $-b_1$ are admissible values.} 
\be\label{b1}
b_1 := \Big( \frac{-k!}{f^{(k_0)}(b_0)}\Big)^{\frac 1{k_0-1}}, \qquad \tilde c_j:= f'(b_0) -c_j^0 \in \R,
\ee
and $\tilde u(t,y) = Q_{c_j}(y-c_j t)$ is a soliton solution of the {\bf focusing} gKdV equation
\be\label{fgKdV}
\tilde u_t + (\tilde u_{yy} + \tilde f(\tilde u))_y =0, 
\ee
where
\be\label{tff}
\tilde f(s) := -\frac 1{b_1} \big[ f(b_0 + b_1s) -f(b_0) - b_1 f'(b_0) s \big].
\ee
\end{thm}

\begin{proof}
See Section \ref{2.5}.
\end{proof}

\noindent
{\bf Remarks.}

\noindent
1. Let $b_0\in \R-\{ 0\}$, and $f_{b_0}(s) := s^5 -5b_0 s^4 +10 b_0^2 s^3 -10 b_0^3 s^2$. Then $c^*(f_{b_0},b_0) =+\infty$ and  $\tilde f_{b_0}(s) =s^5$. Therefore, according to \cite[Theorem 1]{Martel1}, the conclusion of the above Theorem are still valid for the $L^2$-critical case. Even better, the above result can be adapted in the case of $L^2$-supercritical nonlinearities, as a consequence of C\^ote-Martel-Merle \cite{CMM} and Combet \cite{Com}.

\noindent
2. Note that from $k_0\geq 2$ and {\bf(a)} above one has at least $\tilde f \in C^3(\R)$, a sufficient condition to obtain global well-posedness for (\ref{fgKdV})-(\ref{tff}) in $H^1(\R)$ \cite{KPV,MMan}.

\medskip

Following the Martel-Merle-Tsai's paper \cite{MMT}, we say that an initial configuration $u_0\in b_0+H^1(\R)$ for (\ref{dgKdV}) is {\bf well-prepared} if for $L,\al>0$ and  $0<c_1^0<\ldots <c_N^0<c^*$ ($c^*$ given by Theorem \ref{Geven}), and $ x_1^0 <  x_{2}^0<\ldots <  x_N^0 \in \R$, one has
\be\label{Inv22a}
 \big\|  u_0 -b_0 - b_1\sum_{j=1}^N Q_{c_j^0} (\cdot - x_j^0 ) \big\|_{H^1(\R)} \leq \al, \qquad   x_{j}^0 >  x_{j-1}^0 + L, \quad j=2, \ldots, N,
\ee
with $b_1$ defined in (\ref{b1}), and $Q_{c_j^0}$ solitons of (\ref{fgKdV})-(\ref{tff}). Our next result states that this configuration is preserved for positive times.

\noindent

\begin{thm}[Stability of the even multi-kinks]\label{EvenSta}~

There exist $L_0, \al_0>0$ such that for all $L > L_0$ and $ \al \in (0, \al_0)$, a well-prepared initial configuration $u_0$ for (\ref{dgKdV}), satisfying (\ref{Inv22a}), is $H^1$-stable for all positive times. 
\end{thm}

\begin{proof}
The proof is similar to that of Theorem \ref{MT1}. We skip the details.
\end{proof}

Another striking consequence of (\ref{CGg}) is the fact that we can describe the {\bf interaction} among even kinks in some regimes, in the spirit of \cite{MMcol1,MMcol2,Mu1}. Indeed, one has the following

\begin{thm}[Inelastic interaction of even 4-kinks]\label{Collision}~

Let $b_0\in \R$ such that $(a)$-$(b)$ are satisfied, and suppose in addition that $f$ is of class $C^{k_1+1}(\R)$, where  
\be\label{nonint}
f^{(k_1)}(b_0) \neq 0 \quad \hbox{ for some } \; k_1\geq 4.
\ee
Let $c^*=c^*(f,b_0)>0$ be the corresponding threshold for the existence and stability of single solitons for (\ref{fgKdV})-(\ref{tff}). Consider $0<c_1^0 \ll c_2^0\ll c^*$, and let $U_e(t)$ be the unique 4-kink solution of (\ref{dgKdV}) satisfying (\ref{Pure}). Then $U_e(t)$ is global-in-time $H^1$-stable, but it is not pure as $t\to -\infty$.
\end{thm}

\noindent
{\bf Remarks.} 

\noindent
1. The condition (\ref{nonint}) allows us to rule out the integrable cases $f(s) =\al s^2+\beta s^3$, $\al, \beta\in \R$.

\noindent
2. By \emph{not pure as $t\to -\infty$} in Theorem \ref{Collision} we mean that (\ref{lim1}) cannot happen: for any $x_j^-\in \R$,
$$
\liminf_{t\to - \infty}  \big\| U_{e} (t) -b_0 - \sqrt{\beta} \sum_{j=1}^{2} Q_{c_j^0,\beta} (\cdot + \tilde c_j t + x_j^{-} ) \big\|_{H^1(\R)} >0.
$$

\noindent
3. The collision problem has been recently considered in the case of the NLS equation: see e.g. Holmer-Marzuola-Zworski \cite{HMZ1,HMZ2}, Perelman \cite{Pe}, and references therein.

\medskip

\begin{proof}[Proof of Theorem \ref{Collision}]
We prove Theorem \ref{Collision} in Section \ref{2.5}. The main idea is that condition (\ref{nonint}) is the key point to invoke \cite{MMcol1} and our result \cite[Theorem 1.3]{Mu1} to equations (\ref{fgKdV})-(\ref{tff}), classifying the nonlinearities for which the 2-soliton collision is inelastic.
\end{proof}

\medskip

Another collision result is the following remarkable consequence of the recent Martel-Merle's papers describing the interaction of $(i)$ two very different \cite{MMcol1}, and $(ii)$ two nearly equal solitons of the quartic gKdV equation \cite{MMcol3}.

\begin{cor}[Inelastic interaction of even 4-kink solutions, quartic case]\label{Collision2}~

Let $b_0\in \R$, and $f_{b_0}(s) := s^4 -4b_0s^3 +6b_0^2 s^2$. Then $f_{b_0}$ satisfies $(a)$-$(b)$ above, one has $c^*(f_{b_0},b_0) =+\infty$ and for {\bf any} $0<c_1^0 \ll c_2^0 $, the corresponding 4-kink $U_e(t)$ constructed in Theorem \ref{Geven}, and pure as $t\to +\infty$, is globally $H^1$-stable, but it is not pure as $t\to -\infty$. The same result is valid in the regime $0<c_1^0 <c_2^0$,  with $|c_1^0- c_2^0|\ll 1$.
\end{cor}

\noindent
{\bf Remark.}
The last result is a consequence of the fact that from (\ref{fgKdV}) and (\ref{tff}), one has $\tilde f_{b_0}(s) = s^4,$ for which solitons $Q_c$ exist for any $c>0$. Note in addition that for any $b_0\neq 0$, the corresponding nonlinearity $f_{b_0}$ does not allow to perform the standard transformation $u\to -u$, which links the defocusing and focusing quartic equations. Moreover, the quadratic term in $f_{b_0}$ is always of defocusing nature. Therefore, the above result is completely new for $b_0\neq 0$.  

\bigskip

Finally, we consider the case of generalized, odd multi-kink solutions. First of all, we have to recall some important facts. For more details, the reader may consult the monograph of Zhidkov \cite{Z}.

\medskip

Let $\varphi^- , \varphi^+ \in \R$, with $\varphi^-<\varphi^+$, and let $c>0$, $x_0\in \R$ be fixed numbers. Let $f$ be the nonlinearity considered in (\ref{dgKdV}). Suppose that the following hypotheses hold:
\ben
\item[\bf (c)] one has 
$$
c\varphi^- -f(\varphi^-) = c\varphi^+ -f(\varphi^+);
$$
\item[\bf (d)] the function $\displaystyle{F(s) := \int_{\varphi^-}^s (c\sigma -f(\sigma) -c\varphi^- + f(\varphi^-))d\sigma}$ satisfies
$$
F(\varphi^+) =0, \quad \hbox{ and } \quad F(s)<0, \quad \hbox{ for all } s\in (\varphi^-,\varphi^+);
$$
\item[\bf (e)]  $f'(\varphi^\pm) >c$ (non degeneracy condition).\footnote{This condition ensures that the continuous spectrum of the linearized operator $\mathcal L := -\partial_x^2 - c + f'(\varphi_c)$ is bounded from below, away from zero, and the kernel of $\mathcal L$ is spawned by its ground state $\varphi_c'>0$.}
\een
Then there exists a monotone, generalized kink solution of (\ref{dgKdV}), of the form
\be\label{Gkink}
u(t,x) := \varphi_c (x+ ct + x_0),  \qquad \lim_{s\to \pm\infty} \varphi_c(s) =\varphi^{\pm},
\ee
and $\varphi_c$ satisfies
$$
\varphi_c'' + c\varphi_c -f(\varphi_c) = c\varphi^- - f(\varphi^-), \quad \varphi_c'>0, \; \varphi_c' \in H^1(\R).
$$
Moreover, this solution satisfies, for some constants $K,\ga>0$, the following estimates
$$
|\varphi_c (s)- \varphi^\pm| + |\varphi_c' (s)| \leq Ke^{-\ga|s|}.
$$
Finally, but not least important, the condition $\varphi_c'>0$ implies that (\ref{Gkink}) is $H^1$-stable (Zidkhov \cite[p. 91]{Z}, Merle-Vega \cite{MV}). 

\medskip

\noindent
{\bf Remarks.}

\noindent
1. For the sake of clarity, let us mention that in the integrable case $f(s)=s^3$, given $\varphi^-\in \R$ and $c>0$, one has that for $\varphi^+>\varphi^-$, conditions {\bf (c)}-{\bf (d)} and {\bf (e)} lead to necessary conditions $\varphi^- =-\sqrt{c}$ and $\varphi^+ = \sqrt{c} $, namely (\ref{Sol}).

\medskip

\noindent
2. It is important to point out that the multi-kink solution $U_e$ constructed in Theorem \ref{Geven} cannot be decomposed as the sum of several kinks of the form (\ref{Gkink}), at least in a general situation (compare e.g. with (\ref{2K})). Therefore, we believe that (\ref{Pure}) and (\ref{PureOdd}) below are the correct ways to define generalized multi-kink solutions, in the case of defocusing gKdV equations.   

\medskip

\noindent
3. The Cauchy problem associated to (\ref{dgKdV}) with initial condition satisfying $u(0) - \varphi_c  \in H^1(\R)$, is locally well-posed in the class $\varphi_c(\cdot +ct) +  H^1(\R)$. This result is consequence of the analysis carried out by Merle and Vega in \cite{MV} and the fact that $f$ is regular enough. In what follows, we will only consider stable solutions, then globally well defined.

\bigskip

Our next objective is to generalize the Zhidkov's results to the case of $(2N-1)$-kinks, as follows:

\medskip

\begin{thm}[Existence and uniqueness of generalized odd multi-kinks]~

Let $N\geq 2$ and $\varphi^-, \varphi^+ \in \R$ be such that ${\bf(a)}$-${\bf(b)}$ hold with $b_0 := \varphi^-$, and ${\bf(c)}$-${\bf(e)}$ above are satisfied. There exists $c^*:=c^*(f, \varphi^-)>0$ such that, for all $0<c_1^0<c_2^0<\ldots <c_{N}^0<c^*$, and $x_1^-, x_2^-, \ldots, x_N^-\in \R$, the following holds. There exists a unique solution $U_o=U_o(t) $ of (\ref{dgKdV}) such that $u(t) - \varphi_c(\cdot + ct) \in  H^1(\R)$, and it satisfies
\be\label{PureOdd}
\lim_{t\to +\infty} \big\|U_o(t) - \varphi_c(\cdot + ct + x_N^-) -b_1\sum_{j=1}^N Q_{c_j^0}(\cdot + \tilde c_j t + x_j^-) \big\|_{H^1(\R)} =0,
\ee
with $b_1, \tilde c_j$ as in (\ref{b1}), and $\tilde u(t,y) = Q_{c_j}(y-c_j t)$ is a soliton of the focusing gKdV equation (\ref{fgKdV})-(\ref{tff}).
\end{thm}

The proof of this result follows the lines of the proof of Theorem \ref{EoMK}, see also Remark 2 in \cite{Martel1}. Let us recall that in this situation, and following the notation of (\ref{fgKdV})-(\ref{tff}), equation (\ref{eqtu}) now becomes
$$
\tilde u_t + (\tilde u_{yy} + \tilde f(\tilde u))_y = [F(t,y)]_y,  \quad \tilde u\in H^1(\R),
$$
with
\bee
F(t,y) & := & \frac 1{b_1}[ f(\varphi_c + b_1 \tilde u) -f(\varphi_c)  -f(\varphi^- + b_1\tilde u) + f(\varphi^-) ] \\
& = & \frac 1{b_1}\int_0^{b_1\tilde u} \int_{\varphi^-}^{\varphi_c} f''(t+s)dsdt.
\eee
Therefore, since $u\in L^\infty(\R)$, one has $F(t,y) = O(|\varphi_c -\varphi^-||\tilde u|)$, which is enough to conclude.

\medskip 

Finally, we say that an initial configuration $u_0$, perturbation of a kink solution $\varphi_c$, is {\bf well-prepared} if for $L, \al>0$, $0<c_1^0<\ldots <c_{N-1}^0<c^*$, and $ x_1^0 <  x_{2}^0<\ldots <  x_N^0$, one has
\be\label{Inv33}
\big\|  u_0 - \varphi_c(\cdot - x_N^0) - b_1\sum_{j=1}^N Q_{c_j^0} (\cdot - x_j^0 ) \big\|_{H^1(\R)} \leq \al, \quad  x_{j}^0 >  x_{j-1}^0 + L, \quad j=2, \ldots, N,
\ee
with $b_1$ defined in (\ref{b1}). In addition, by taking $c^*$ smaller if necessary, we assume that each soliton $ Q_{c_j^0}$ is stable in the sense of Weinstein (\ref{WC}).

\begin{thm}[Stability of the odd multi-kinks]\label{OddSta}~

There exist $L_0, \al_0>0$ such that for all $L>L_0$ and $\al\in (0,\al_0)$, a well-prepared initial data $u_0$ for (\ref{dgKdV}), satisfying (\ref{Inv33}), is $H^1$-stable for all positive times.
\end{thm}

This result is proved following the lines of the proof of Theorem \ref{MT2}, using in addition that single solitons are stable. We skip the details. 

\bigskip

\noindent
{\bf Final remarks.}

\noindent
1. We recall that the collision problem in the case of odd multi-kink solutions remains an interesting open question. In addition, we believe that our approach introduces new ideas to deal with the dynamics of kink solutions in the $L^2$-critical and supercritical setting, by using a suitable focusing counterpart.

\smallskip

\noindent
2. Let us mention that a similar transformation to (\ref{CG2}) can be introduced in the cases of the $\phi^4$ and sine-Gordon models, with different results. For the first equation,
\be\label{p4}
u_{tt}-u_{xx} = u(1-u^2), \quad u(t,x) \in \R,
\ee
it is well know that $u(t,x) = \varphi(x)$, with $\varphi$ given in (\ref{Q}), is a stationary kink solution. However, the transformation $u(t,x) := \varphi(x) - \frac 23 \tilde u(\sqrt 2 t,\sqrt 2 x)$  
leads to the following equation for $\tilde u(s,y)$
$$
\tilde u_{ss}- \tilde u_{yy} + \tilde u - \tilde u^2 +  \frac 29 \tilde u^3 = \frac 32(1-\varphi^2)\tilde u - (1-\varphi)\tilde u^2.
$$
Looking for an approximate, localized and stationary solution, we arrive to study the elliptic equation associated to the Gardner nonlinearity:
$$
\tilde u_{yy} -\tilde u + \tilde u^2 - \frac 29 \tilde u^3=0, \quad \tilde u\in H^1(\R).
$$
This is a Gardner elliptic equation with parameters $c:=1$ and $\beta := \frac 29$ (cf. (\ref{Id1})), for which the only localized, positive solution is zero. Therefore it is not possible (at least formally) to attach to the kink solution of (\ref{p4}) suitable soliton-like structures of the form (\ref{SolG}). 

Now we perform the same analysis in the case of the sine-Gordon equation
\be\label{SG}
u_{tt} - u_{xx} +\sin u =0, \quad u(t,x) \in \R,
\ee
and its kink solution $\varphi(x) := 4\arctan e^x$. Indeed, using the transformation $u(t,x): = \varphi(x) + \tilde u(t,x)$, we arrive to the following perturbed, sine-Gordon equation
$$
\tilde u_{tt}-\tilde u_{xx} + \sin \tilde u = (1-\cos \varphi) \sin \tilde u +(1-\cos v)\sin \varphi,
$$
where the right hand side is small if we consider $\tilde u$ as a localized solution of (\ref{SG}) in the region where $\varphi \sim 0$. We can put for instance, a sum of \emph{breather} solutions, provided this solution is stable, which is an open problem. We expect to consider some of these problems in a forthcoming publication.

\bigskip

\subsection*{Idea of the proofs} Theorems \ref{UniE}, \ref{MT1}, \ref{Geven} and \ref{EvenSta} can be deduced from Martel \cite{Martel1} and Martel-Merle-Tsai \cite{MMT,MMan}. We recall that, without using transformation \ref{CG},  these results were unable to be tackled down by using any direct method.

\medskip

We prove Theorem \ref{MT2} in Section \ref{3}. The proof is based in the approach introduced in \cite{MMT} in order to describe the stability in $H^1(\R)$ of $N$ decoupled solitons. However, in this opportunity we face several new problems since the kink solution and the Gardner solitons are in strong interaction through the dynamics. Moreover, the mass (\ref{M}) cannot be used to control the Gardner solitons, as has been done in \cite{MMT}. This means that Theorem \ref{MT2} cannot be deduced from the standard Zhidkov \cite{Z} and Martel-Merle-Tsai \cite{MMT} results, and we need new ideas. In that sense, the transformation (\ref{CG2}) is the first step --and the more important one-- to understand the interaction among kinks as actually localized, soliton-like interactions.

\medskip

Let us be more precise. Using the energy (see (\ref{E})) of the solution $u(t)$, one can control with no additional difficulties the kink solution. This is a consequence of the non negative character of the linearized operator around the kink solution, see \cite{Z,MV} for more details. However, this quantity is far from being enough to control the behavior of the Gardner solitons. We overcome this difficulty by using the transformation (\ref{CG2}), which introduces a new function $\tilde u(t)$, almost solution of a Gardner-like equation (cf. (\ref{eqtu})). It turns out that the perturbative terms on the right hand side of (\ref{eqtu}) can be controlled provided the solitons are far from the center of the main kink solution, which holds true if we assume that  the initial configuration is well prepared (see Proposition \ref{Pr0}). Additionally, we introduce a new, {\bf almost conserved mass} (see (\ref{Mjj})) for the portion on the left of the solution $\tilde u$,  which allows to control each Gardner soliton by separated. Using this, we avoid the problem of using the natural mass (\ref{M}), which is very bad behaved for $H^1(\R)$ perturbations. This approach is completely general and can be adapted to prove Theorem \ref{OddSta}. No additional hypotheses are needed, only the single stability of each generalized soliton component of the multi-kink solution. The proof of the asymptotic stability property generalizes the argument used \cite{MMan}, this time to the function $\tilde u$.

\medskip

Finally, concerning Theorem  \ref{EoMK} --proved in Section \ref{4}--, we extend the result of Martel \cite{Martel1}. Most of the proof is similar to the proof of Theorem \ref{MT2}, but estimates are easier to carry out since we do not need to control the scaling parameters of each Gardner solitons.

\bigskip

\section{Proof of Theorems \ref{UniE} and \ref{MT1}}\label{2}

\medskip

\noindent
{\bf Proof of Theorem \ref{UniE}}.
Let $\tilde U$ be another solution of (\ref{mKdV}) satisfying (\ref{lim1}). Then,  from (\ref{CG}),
$$
\tilde V(t,y):= \frac{1}{ \sqrt{\beta}} \Big[ b -\tilde U(t, y -\frac t{3\beta}) \Big],
$$
is solution of the Gardner equation (\ref{Ga}) and satisfies (\ref{Minus}). From the uniqueness of $V^{(N)}$ \cite{Martel1}, one has $\tilde V \equiv V^{(N)}$, and therefore $\tilde U \equiv U_{e}$.

\medskip

\noindent
{\bf Proof of Theorem \ref{MT1}}.
First of all, let us recall the Martel-Merle-Tsai's stability result \cite{MMT}:

\begin{thm}[$H^1$-stability of the sum of $N$-Gardner solitons, \cite{MMT,MMan}]\label{MT0}~

Let $N\geq 2$, $\beta>0$ and $0<c_1^0<c_2^0<\ldots <c_N^0<\frac{2}{9\beta}$ be such that (\ref{WC}) holds for all  $j=1,\ldots, N.$ There exists $\tilde \al_0, \tilde A_0, \tilde L_0, \tilde\ga>0$ such that the following is true.  Let $v_0 \in H^1(\R)$, and assume that there exists $\tilde L> \tilde L_0$, $\tilde \al \in (0,\tilde \al_0)$ and $\tilde x_1^0< \tilde x_2^0<\ldots < \tilde x_N^0$, such that 
\be\label{Inv}
 \big\|  v_0 -\sum_{j=1}^N Q_{c_j^0,\beta} (\cdot - \tilde x_j^0 ) \big\|_{H^1(\R)} \leq \tilde \al,  \qquad \tilde x_{j}^0 > \tilde x_{j-1}^0 + \tilde L, \quad j=2, \ldots, N.
\ee
Then there exists $\tilde x_1(t), \ldots \tilde x_N(t)$ such that the solution $v(t)$ of the Cauchy problem associated to (\ref{Ga}), with initial data $v_0$, satisfies 
$$
v(t) = S(t) + w(t), \quad S(t) :=\sum_{j=1}^N Q_{c_j^0,\beta}(\cdot - \tilde x_j(t) ),
$$
and
\be\label{Fnv}
\sup_{t\geq 0} \Big\{ \| w(t)\|_{H^1(\R)} + \sum_{j=1}^N  |\tilde x'_j(t)-c_j^0|   \Big\} \leq \tilde A_0 (\tilde \al+ e^{-\tilde \ga \tilde L}).
\ee
Moreover, there exist $c_j^\infty>0$ such that $\lim_{+\infty} \tilde x_j'(t) =c_j^\infty$ and
\be\label{AAAA}
\lim_{t\to +\infty} \| v(t) - \sum_{j=1}^N Q_{c_j^\infty,\beta} (\cdot - \tilde x_j(t))\|_{H^1(x>\frac  {c_1^0}{10}t)} =0.
\ee
\end{thm}

\medskip

It is important to stress that the well-preparedness restriction on the initial data (\ref{Inv}) is by now necessary since there is no satisfactory collision theory for the non-integrable cases.\footnote{See \cite{MMcol1,MMcol2,MMcol3,Mu1} for some recent results describing the collision of two solitons for gKdV equations in some particular regimes and with general nonlinearities, beyond the integrable cases.}   

\medskip

However, as explained in \cite{MMT} for the KdV case, the above argument can be extended to a global-in-time stability result, thanks to the continuity of the Gardner flow in $H^1(\R)$ \cite{KPV}, and the fact that the Gardner equation (\ref{Ga}) is an integrable model, with explicit $N$-soliton solutions (see (\ref{Minus})-(\ref{Plus})), given by the family $V^{(N)}$ above described. Therefore, a direct consequence of this property and the invariance of the equation under the transformation $u(t,x) \mapsto u(-t,-x)$ is the following

\medskip

\begin{cor}[$H^1$-stability of Gardner multi-solitons, \cite{MMT,MMan}]\label{StaG}~

Let  $\delta>0$, $N\geq 2$, $0<c_1^0<\ldots<c_N^0$ and $x_1^0,\ldots,x_N^0 \in \R$. There exists $\al_0>0$ such that if $0<\al<\al_0$, then the following holds. Let $v(t)$ be a solution of (\ref{Ga}) such that 
$$
\| v(0)- V^{(N)}(\cdot; c_j^0, x_j^0) \|_{L^2(\R)}\leq \al,
$$
with $V^{(N)}$ the $N$-soliton satisfying (\ref{Minus})-(\ref{Plus}). Then there are $x_j(t)\in \R$, $j=1,\ldots, N$, such that
\be\label{stabilitycor}
\sup_{t\in \R} \big\| v(t) - V^{(N)}( \cdot ; c_j^0, x_j(t)) \big\|_{H^1(\R)} \leq \delta.
\ee
Moreover, there exist $c_j^{\infty}>0$  such that
\be\label{asympstabcor}
\lim_{t\to +\infty} \big\| v(t) - V^{(N)}(\cdot ; c_j^{\infty},x_j(t)) \big\|_{H^1(x> \frac {c_1^0}{10} t )} =0, 
\ee
and $x_j(t)$ are $C^1$ for all $|t|$ large enough, with $x_j'(t)\to -c_j^{\infty}\sim -c_j^0$ as $t\to +\infty$. A similar result holds as $t\to -\infty$, with the obvious modifications.
\end{cor}

\noindent
{\bf Remark.}
Let us emphasize that the proof of this result requires the \emph{existence} and the \emph{explicit behavior} of the multi-soliton solution $V^{(N)}$ of the Gardner equation, and therefore the integrable character of the equation. In particular, we do not believe that a similar result is valid for a completely general, non-integrable gKdV equation, unless one considers some perturbative regimes (cf. \cite{MMcol1,MMcol3} for some global $H^1$-stability results in the non-integrable setting.)

\medskip

Therefore, using (\ref{CG}) and the previous result one has the following more precise version of Theorem \ref{MT1}.

\begin{thm}[Stability of \emph{even} multi-kink solutions]\label{MT1b}~

The family of multi-kink solutions $U_{e} (t)$ from Definition \ref{MKe} and (\ref{Unbe}) is global-in-time $H^1$-stable, and asymptotically stable as $t\to \pm \infty$. More precisely, let  $\beta, \delta >0$, $N\geq 2$, $0<c_1^0<\ldots<c_N^0<\frac 2{9\beta}$ and $x_1^0,\ldots,x_N^0\in \R$. There exists $\al_0>0$ such that if $0<\al<\al_0$, then the following holds. Let $u(t)$ be a solution of (\ref{mKdV}) such that 
\be\label{IC}
\| u(0)- U_{e} (\cdot; c_j^0, x_j^0) \|_{H^1(\R)}\leq \al,
\ee
with $U_{e} $ the $2N$-kink solution defined in (\ref{Unbe}). Then there exist $x_j(t)\in \R$, $j=1,\ldots, N$, such that
\be\label{stacor}
\sup_{t\in \R} \big\| u(t) - U_{e} ( \cdot ;c_j^0,x_j(t)) \big\|_{H^1(\R)} \leq \delta.
\ee
Moreover, there exist $c_j^{\infty}>0$  such that
\be\label{astabcor}
\lim_{t\to + \infty} \big\| u(t) - U_{e} (\cdot ; c_j^{\infty},x_j(t)) \big\|_{H^1(x> (\frac {1}{10} c_1^0 -3c)t )} =0, 
\ee
and $x_j(t)$ are $C^1$ for all $|t|$ large enough, with $x_j'(t)\to c_j^{\infty}\sim \tilde c_j$ as $t\to +\infty$. A similar result holds as $t\to -\infty$, with the obvious modifications.
\end{thm}

\noindent
{\bf Remark}.
Let us recall, for the sake of completeness, that estimate (\ref{astabcor}) is deduced from (\ref{asympstabcor}) by using the transformation (\ref{CG}).

\bigskip

\section{Proof of Theorems \ref{Geven} and \ref{Collision}}\label{2.5}

\medskip

\noindent
{\bf Proof of Theorem \ref{Geven}}.
Thanks to {\bf(a)}-{\bf(b)}, there exists a generalized transformation of the form (\ref{CG}), such that
\be\label{CGg}
u(t,x) = b_0 + b_1\tilde u(t,x+ f'(b_0)t), 
\ee
with $b_1$ given by (\ref{b1}), and such that $\tilde u(t,y)$ satisfies (\ref{fgKdV})-(\ref{tff}). Moreover, note that a Taylor expansion gives us that $\tilde f$ is a subcritical perturbation of the pure power nonlinearity:
\bea
\tilde f(s)&  = & -\frac{1}{k_0!}b_1^{k_0-1} f^{(k_0)}(b_0)s^{k_0}  - \frac1{(k_0+1)!} b_1^{k_0}  f^{(k_0+1)}(\xi)s^{k_0+1}  \nonu \\
& =& s^{k_0} + \tilde f_{b_0}(s), \quad k_0\in \{ 2,3, 4\}, \label{expansion}
\eea
for some $\xi$ in between $b_0$ and $b_0 + b_1s$. Note in addition that
$$
\lim_{s\to 0}\frac{\tilde f_{b_0}(s)}{|s|^{k_0}} =0.
$$
According to Berestycki and Lions \cite{BL}, $\tilde f$ is an admissible nonlinearity for the {\bf existence of small solitons}, in the sense that there exists $c^*>0$ (depending on $f$ and $b_0$ fixed), such that for all $0<c<c^*$, there exists a solution $\tilde u=\tilde u(t,y)$ of (\ref{fgKdV})-(\ref{tff}), of the form
$$
\tilde u(t,y) = Q_c(y-ct),
$$
and such that $Q_c=Q_c(s)$ satisfies
$$
Q_c'' -c Q_c + \tilde f(Q_c) =0, \quad Q_c>0, \quad Q_c\in H^1(\R).
$$
Moreover, $Q_c$ can be chosen even and exponentially decreasing as $s\to \pm \infty$. 

\medskip

In addition, for $0<c<c^*$ small, solitons satisfy the corresponding Weinstein condition (\ref{WC}) (cf. Martel-Merle \cite{MMan}), which implies orbital stability in the energy space $H^1(\R)$. From Theorem  in \cite{Martel1}, given $0<c_1^0<\ldots < c_N^0<c^*$, $N\geq 2$, there exists a unique solution $\tilde U \in C(\R, H^1(\R))$ of (\ref{fgKdV})-(\ref{tff}), satisfying 
$$
\lim_{t\to +\infty} \Big\|\tilde U(t) - \sum_{j=1}^N Q_{c_j^0}(\cdot - c_j^0 t) \Big\|_{H^1(\R)} =0.
$$
The final conclusion follows after applying the transformation (\ref{CGg}). Note that $\tilde c_j$ defined in (\ref{b1}) can be either zero, positive or negative, depending on $b_0$ and $c^*$; however one always has $\tilde c_N<\tilde c_{N-1} <\ldots <\tilde c_1.$

\medskip

\noindent
{\bf Proof of Theorem \ref{Collision}}.
Let us consider the transformation (\ref{CGg}), which leads to the focusing gKdV equation (\ref{fgKdV})-(\ref{tff}). Let $k_1\geq 4$ be the first integer satisfying $f^{(k_1)}(b_0)\neq 0$. Note that from (\ref{b1}) and (\ref{expansion}), one has 
$$
\tilde f(s) = 
\begin{cases}  
\displaystyle{s^{2}  - \frac1{3!} b_1^{2}  f^{(3)}(b_0)s^{3}  - \frac1{k_1!} b_1^{k_1-1}  f^{(k_1)}(b_0)s^{k_1} + \tilde f_1(s), \quad k_0=2;} \\
  \displaystyle{ s^{3}  - \frac1{k_1!} b_1^{k_1-1}  f^{(k_1)}(b_0)s^{k_1}  + \tilde f_1(s), \quad k_0 =3;} \\ 
\displaystyle{s^{4}  + \tilde f_1(s), \quad k_0=4.}
\end{cases}
$$
Note that in each case one has
$$
\lim_{s\to 0}\frac{\tilde f_{1}(s)}{|s|^{k_1}} =0.
$$
Since by hypothesis $ f^{(k_1)}(b_0)\neq 0$, one has that $\tilde f$ is a nontrivial perturbation of the integrable models $\tilde f(s) =s^2, s^3$ and $\tilde f(s) =s^2 +\beta s^3$.

Therefore, from the classification theorem for the regime $0<c_1\ll c_2 \ll c^*$ showed in \cite{MMcol1,MMcol2,Mu1}, one can conclude that the 2-soliton structure is globally $H^1$-stable, but the solution $U_e$ constructed in Theorem \ref{Geven} is {\bf never pure} as $t\to -\infty$. The final conclusion follows after applying (\ref{CGg}). The proof is complete.

\bigskip

\section{Proof of Theorem \ref{MT2}}\label{3}

\medskip

In this section we prove Theorem \ref{MT2}. First of all, we state a more detailed version of this result.

\begin{thm}[Stability of odd multi-kink solutions]\label{MT2b}~

The family of multi-kink solutions $U_{o} (t)$ from Definition \ref{MKo} is global-in-time $H^1$-stable, and asymptotically stable as $t\to \pm \infty$.  More precisely, let  $\delta, \beta >0$, $N\geq 2$, $0<c_1^0<\ldots<c_{N-1}^0<\frac 2{9\beta}$ and $x_1^0,\ldots,x_N^0\in \R$. There exists $\al_0>0$ such that if $0<\al<\al_0$, then the following holds. Let $u(t)$ be a solution of (\ref{mKdV}) such that 
\be\label{IC2}
\| u(0)- U_{o} (\cdot; c_j^0, x_j^0) \|_{H^1(\R)}\leq \al, \quad c_N^0 := \frac 1{9\beta}
\ee
with $U_{o} $ the $(2N-1)$-kink solution from Definition \ref{MKo}. Then there exist $x_j(t)\in \R$, $j=1,\ldots, N$, such that
\be\label{stacor2}
\sup_{t\in \R} \big\| u(t) - U_{o} ( \cdot ;c_j^0,x_j(t)) \big\|_{H^1(\R)} \leq \delta.
\ee

\end{thm}

\medskip

It turns out that the proof of Theorem \ref{MT2b} follows as a consequence of the following

\begin{prop}[$H^1$-stability of the one kink and $(N-1)$ Gardner solitons]\label{Pr0}~

Let $N\geq 2$, $\beta>0$, $c_N^0=\frac 1{9\beta}$, and $0<c_1^0<c_2^0<\ldots <c_{N-1}^0<\frac{2}{9\beta}$ be such that (\ref{WC}) holds for all  $j=1,\ldots, N.$ There exists $\al_0, A_0,  L_0, \sigma_0>0$ such that the following is true.  Let $u_0 \in H^1(\R)$, and assume that there exists $ L>  L_0$, $ \al \in (0, \al_0)$ and $ x_1^0< x_2^0<\ldots <  x_N^0$, such that 
\be\label{InvG}
 \big\|  u_0 - \varphi_{c_N^0}(\cdot + x_N^0) - \sqrt{\beta}\sum_{j=1}^{N-1} Q_{c_j^0,\beta} (\cdot +  x_j^0 ) \big\|_{H^1(\R)} \leq \al, \quad  x_{j-1}^0 <  x_{j}^0 -  L,
\ee
for $j=2, \ldots, N$. Then there exists $x_1(t), \ldots,  x_N(t)$ such that the solution $u(t)$ of the Cauchy problem associated to (\ref{mKdV}), with initial data $u_0$, satisfies 
$$
 u(t) = S(t) + w(t), \qquad  S(t) :=\varphi_{c_N^0}(\cdot + c_N^0 t +x_N(t)) + \sqrt{\beta}\sum_{j=1}^{N-1} Q_{c_j^0,\beta}(\cdot + \tilde c_j t+ x_j(t) ),
$$
and
\be\label{FnvG}
\sup_{t\geq 0} \Big\{ \| w(t)\|_{H^1(\R)} + \sum_{j=1}^N  | x'_j(t)|   \Big\} \leq  A_0 (\al+ e^{- \sigma_0  L}).
\ee
\end{prop}

\medskip

\noindent
{\bf Proof of Theorem \ref{MT2b}.} From Proposition \ref{Pr0}, the proof of Theorem \ref{MT2b} follows directly from the integrable character of the mKdV equation and the existence of a suitable multi-kink solution satisfying (\ref{lim3})-(\ref{lim4}). See e.g. \cite[Corollary 1.2]{AMV} for a similar, detailed proof.

\medskip

Therefore, we are left to prove Proposition \ref{Pr0}. 

\medskip

\noindent
{\bf Proof of Proposition \ref{Pr0}.}

\smallskip

\noindent
{\bf Stability.} Let us assume the hypotheses of Proposition \ref{Pr0}. Let $\sigma_0$ satisfying
\be\label{sig0}
0<\sigma_0\leq \frac 12\min(c_2^0-c_1^0,c_3^0-c_2^0,\ldots, 2 c_N^0 -c_{N-1}^0),
\ee
a measure of the minimal difference among the scaling parameters. This quantity may change from one line to another, but always satisfies (\ref{sig0}).

\smallskip

Note that a simple continuity argument, using the local Cauchy theory developed in \cite{MV} shows that there exists $t_0>0$ such that
\be\label{dec1}
 \sup_{t\in [0, t_0]} \Big\| u(t) - \varphi_{c_N^0}(\cdot + c_N^0 t + \tilde x_N(t)) - \sqrt{\beta}\sum_{j=1}^{N-1} Q_{c_j^0,\beta}(\cdot + \tilde c_j t+ \tilde x_j(t)) \Big\|_{H^1(\R)}  \leq 2(\al + e^{-\sigma_0 L}),
\ee
for some $\tilde x_j(t) \in \R$, $j=1,\ldots, N$. Therefore, given $K^*>2$, we can define the following quantity
\bea\label{Tstar}
T^* & := & \sup \big\{ T> 0, \; \hbox{ for all } t\in [0, T], (\ref{dec1}) \hbox{ is satisfied with 2 replaced by $K^*$}, \nonu\\
& & \qquad \qquad   \hbox{ and for some } \tilde x_j(t)\in \R. \big\}.
\eea
Our objective is to show that for some $K^*>0$ large enough, one has $T^*= +\infty.$ Following a contradiction argument, we will assume $T^*<+\infty$. This allows to prove the following modulation property.

\begin{lem}[Modulation]\label{dec}~

Possibly taking $\al_0>0$ and $\frac 1{L_0}$ smaller, there exists $K>0$ independent of $K^*$, such that if $L>L_0$ and $0<\alpha<\alpha_0$, the following holds.  There exist unique $C^1$ functions $c_j : [0,T^*]\to (0,+\infty)$, $j=1,\ldots, N-1$, and $x_j:[0,T^*]\to \R$, $j=1,\ldots, N$, such that
\be\label{decomp}
z(t,x) := u(t,x) -\varphi_{c_N^0} (x+c_N^0 t+x_N(t)) -\sqrt{\beta} \sum_{j=1}^{N-1} Q_{c_j,\beta} (x + \tilde c_j t + x_j(t)), 
\ee
satisfies, for all $j=1,\ldots, N$, and for all $t\in [0, T^*]$,
\be\label{Ortho1}
\int_\R z(t,x) \varphi_{c_N^0}'(x+c_N^0 t + x_N(t))dx =0,
\ee
and
\be\label{Ortho2}
\int_\R z(t,x) Q_{c_j,\beta}'(x+ \tilde c_j t+ x_j(t))dx =\int_\R z(t,x) Q_{c_j,\beta} (x+ \tilde c_j t+ x_j(t)) dx=0.
\ee
Moreover, there exists $K>0$ such that for all $t\in [0,T^*]$,
\be\label{smallness1}
\|z(t) \|_{H^1(\R)}+ \sum_{j=1}^{N-1} |c_j(t)-c_j(0)| \leq K K^* (\alpha + e^{-\sigma_0 L}),
\ee
and
\be\label{smallness2}
\|z(0)\|_{H^1(\R)} + \sum_{j=1}^{N-1} |c_j(0)-c^0_j| \leq K\al. 
\ee
\end{lem}
\begin{proof}
The proof of this result is a standard exercise of Implicit Function Theorem, see e.g. Lemma in \cite{MMT} for a similar proof.  Note that in this opportunity we have modulated the translation parameter associated to the kink solution (cf. (\ref{Ortho1})). From \cite{Z,MV}, there is no need to modulate the scaling parameter, see e.g. Lemma \ref{Zi} in Appendix \ref{B}.
\end{proof}

In what follows, we introduce some useful notation. Let us consider
\be\label{R}
R(t,x) := \sqrt{\beta} \sum_{j=1}^{N-1} Q_{c_j,\beta} (x + \tilde c_j t + x_j(t)),  \quad (=\hbox{the Gardner solitons})
\ee
and $\tilde u(t,y)$ defined by the relation
\be\label{tu}
u (t,x)  := \varphi_{c}(x+ c t +x_N(t)) + \sqrt{\beta} \tilde u(t,x+ \frac t{3\beta}),
\ee
where, for the sake of clarity, we have defined $c :=c_N^0$. In particular,
\bea
\tilde u(t,y) &  = & \frac{1}{\sqrt{\beta}}(R + z)(t,y - \frac t{3\beta}) \nonu\\
& = & \sum_{j=1}^{N-1} Q_{c_j,\beta}(y -c_jt +x_j(t)) + \frac{1}{\sqrt{\beta}}z(t,y - \frac t{3\beta})  \nonu\\
& = :&  \tilde R(t,y) + \tilde z(t,y). \label{tRtz} 
\eea
A simple computation shows that $\tilde u=\tilde u(t,y) $ satisfies the modified Gardner equation (compare with (\ref{eqtu}))
\be\label{eqtuV}
\tilde u_t + (\tilde u_{yy} + \tilde u^2 -\beta \tilde u^3)_y = 3[ (\varphi_{c}^2- c)\tilde u + (\sqrt{c} +\varphi_{c}) \tilde u^2]_y +\frac{x_N'(t)}{\sqrt{\beta}} \varphi_c'.
\ee
In this last equation  $\varphi_c$ is a function of the variable $y$ in the sense that $\varphi_c(x+ct + x_N(t)) = \varphi_c (y-2ct + x_N(t)).$ The following result gives an explicit expansion of the energy of $u(t)$.

\begin{lem}[Expansion of the energy]\label{EE0}~

Consider the energy $E[u](t)$ defined in (\ref{E}), for $c=c_N^0 =\frac 1{9\beta}$. Then, for any $t\in [0, T^*]$, one has the following decomposition
\be\label{EE}
E[u](t)   =  E[\varphi_c] +\frac 23 \sum_{j=1}^{N-1}c_j^{3/2}(t) +\mathcal F(t) + O(\|z(t)\|_{H^1(\R)}^3)+ O(e^{-\sigma_0L}),
\ee
with $\mathcal F(t)$ the following second order functional
\be\label{Ft}
\mathcal F(t):= \frac 12\int_\R (z_x^2(t) +2cz^2(t) -3 (c-\varphi_c^2)z^2(t)   - 6 \sqrt{c} R z^2(t) + 3R^2 z^2(t)).
\ee
\end{lem}

\begin{proof}
From (\ref{decomp}) and (\ref{R}), one has:
\bee
E[u](t) &= & E[\varphi_c + R + z](t)\\
& =& E[\varphi_c](t)  + \frac 12 \int_\R R_x^2 + c\int_\R R^2  - \sqrt{c}\int_\R R^3 +\frac 14 \int_\R R^4 \\
& & + \int_\R (\varphi_c + \sqrt{c})R^3  + \int_\R (\varphi_c)_x R_x +\frac 32 \int_\R (\varphi_c^2 -c) R^2+ \int_\R \varphi_c (\varphi_c^2-c)R \\
& & - \int_\R z(\varphi_c'' + c\varphi_c -\varphi_c^3) -\int_\R z(R_{xx} -2c R +3\sqrt{c} R^2 -R^3)  \\
& & + 3\int_\R (\varphi_c^2-c)R z +\frac 12\int_\R (z_x^2 +2cz^2 -3 (c-\varphi_c^2)z^2   - 6 \sqrt{c} R z^2  + 3R^2 z^2) \\
& &  + 3 \int_\R (\varphi_c + \sqrt{c})R z^2+  3 \int_\R R^2(\varphi_c + \sqrt{c}) z  +\int_\R \varphi_c z^3 + \int_\R R z^3 +\frac 14 \int_\R z^4.  
\eee
First of all, note that the term $E[\varphi_c](t)$ actually does not depend on $t$. Additionally, from (\ref{Sol}) and (\ref{ecvarfi}) one has
\be\label{varphic}
\varphi_c'' + c\varphi_c  -\varphi_c^3 =0.
\ee
In order to obtain some estimates of the above quantities, we need the following

\begin{lem}[Identities for $R(t)$]\label{Reqns}~

Let $R$ be the sum of $N$ decoupled Gardner solitons defined in (\ref{R}). Then one has the following identities:
\be\label{ecR}
R_{xx} -2 cR + 3\sqrt{c} R^2 - R^3  =  -\sqrt{\beta} \sum_{j=1}^{N-1}(2c-c_j) Q_{c_j,\beta} + O_{H^1(\R)}(e^{-\sigma_0 L}).
\ee
and
\be\label{ene}
 \frac 12 \int_\R R_x^2 + c\int_\R R^2 - \sqrt{c} \int_\R R^3 +\frac 14 \int_\R R^4  =   \frac 23  \sum_{j=1}^{N-1} c_j^{3/2}(t)  + O(e^{-\sigma_0 L}).
\ee
\end{lem}

\begin{proof}
Let us prove (\ref{ecR}). From (\ref{R}) and the fact that  $c= \frac 1{9\beta}$, one has
\bee
\hbox{ l.h.s. of } (\ref{ecR}) & =&   \sqrt{\beta} \sum_{j=1}^{N-1} ( Q_{c_j,\beta}'' - 2c Q_{c_j,\beta} ) + 3\beta \sqrt{c} \big(  \sum_{j=1}^{N-1} Q_{c_j,\beta}\big)^2  - \beta^{3/2} (  \sum_{j=1}^{N-1} Q_{c_j,\beta}\big)^3 \\
& =&  \sqrt{\beta} \sum_{j=1}^{N-1} ( Q_{c_j,\beta}'' - 2c Q_{c_j,\beta}  + Q_{c_j,\beta}^2 -\beta Q_{c_j,\beta}^3  )  \\
& & \qquad +  \sqrt{\beta} \sum_{i\neq j}^{N-1} Q_{c_i,\beta} Q_{c_j,\beta}   - \beta^{3/2} \Big[ (\sum_{j=1}^{N-1} Q_{c_j,\beta})^3 - \sum_{j=1}^{N-1} Q_{c_j,\beta}^3 \Big] .
\eee
Using the equation for $Q_{c,\beta}$ (cf. (\ref{Id1})), one has
$$
\hbox{ l.h.s. of } (\ref{ecR}) =  -\sqrt{\beta} \sum_{j=1}^{N-1} (2c - c_j) Q_{c_j,\beta}+  O_{H^1(\R)}(e^{-\sigma_0 L}),
$$
as desired.

\medskip

Now we consider (\ref{ene}). From (\ref{R}) and  (\ref{EQb}), one has
\bee
\hbox{ l.h.s. of } (\ref{ene}) & = &  \beta \sum_{j=1}^{N-1} \int_\R \Big\{ \frac 12  Q_{c_j,\beta}'^2 + cQ_{c_j,\beta}^2 - \sqrt{\beta c} Q_{c_j,\beta}^3 +\frac {\beta}4 Q_{c_j,\beta}^4\Big\}  + O(e^{-\sigma_0 L})\\
& = &  \beta \sum_{j=1}^{N-1} \Big\{ E_{\beta} [Q_{c_j,\beta}] + 2c M[Q_{c_j,\beta}] \Big\}  + O(e^{-\sigma_0 L}) \\
& = &  \frac 2{3} \sum_{j=1}^{N-1}   c_j^{3/2}(t)  + O(e^{-\sigma_0 L}). 
\eee
The proof is complete.
\end{proof}

Let us come back to the proof of Lemma \ref{EE0}. From the above results, the orthogonality conditions (\ref{Ortho2}), (\ref{Ft}) and (\ref{varphic}) we have
\bea
E[u](t) & =& E[\varphi_c] +\frac 23 \sum_{j=1}^{N-1} c_j^{3/2}(t)  + \mathcal F(t)  + O (\|z(t)\|_{H^1(\R)}^3) \nonu\\ 
& & -\int_\R (\varphi_c -\sqrt{c})R^3  -\int_\R (\varphi_c)_x R_x +\frac 32 \int_\R (\varphi_c^2 -c) R^2 \label{Eq1} \\
& &-\int_\R \varphi_c (\varphi_c^2-c)R  -6\int_\R (\varphi_c^2-c)R z  - 3 \int_\R (\varphi_c -\sqrt{c})R z^2 .  \label{Eq2}
\eea
Finally, the last two lines in the above identity, namely (\ref{Eq1})-(\ref{Eq2}), are exponentially small. Indeed, one has e.g.
\be\label{est1}
\abs{\int_\R (\varphi_c -\sqrt{c})R^3} \leq  Ke^{-\sigma_0 L}.
\ee
The other terms can be bounded in a similar fashion. From these estimates, (\ref{EE}) follows directly. 
\end{proof}

In the next step, we introduce a modified mass, almost monotone in time, which allows to control the Gardner solitons. Let $Q (s):= (\cosh(s))^{-1}$,
$$
\phi(x) :=m Q(\sqrt{\sigma_0} x/2), \quad \psi(x) :=\int_{-\infty}^x
\phi(s)ds, \quad \hbox{ where }\ m:= \Big[\frac{2}{\sqrt{ \sigma_0 } }
 \int_{-\infty}^\infty
Q \Big]^{-1}. 
$$ 
Note that, for all $x\in \R$, $\psi'(x)>0$,  $0< \psi(x)<1,$ and $\lim_ {x\to -\infty}\psi(x)=0,$
$\lim_ {x\to +\infty}\psi(x)=1.$ 

\medskip

Finally, let, for $j=1,\ldots, N-1$, the {\bf modified mass}
\be\label{Mjj}
M_j(t) : = \frac 12 \int_\R  \tilde u^2(t,y) (1-\psi_j(t,y))\, dy, \qquad  \psi_j(t,x) := \psi(y - \sigma_j(t)), 
\ee
with $\sigma_j(t):=\frac12 (c_{j-1}^0t + c_j^0 t + x_{j-1}^0 + x_{j}^0)$, and $\tilde u$ defined in (\ref{tu}). Note that this quantity considers the mass {\bf on the left} of each soliton, which represents the main difference, compared with the standard arguments included in \cite{MMT,MMnon}.

\medskip

\begin{lem}[Almost monotonicity of the mass, see also \cite{MMT}] \label{th:2}~

There exist $K>0$ and $L_0>0$ such that, for all $L>L_0$,  
the following is true. For all $t\in [0, T^*]$ one has
\be\label{th:2a}
  M_j(t) - M_j(0)\geq -K \, e^{- \sigma_0 L }.
\ee
\end{lem}

\begin{proof}
The proof is similar to \cite{Martel1,MMT}, so we sketch the main steps. Let $j\in \{1,\ldots,N-1\}$.
Using equation (\ref{eqtuV}) and
integrating by parts several times, we have 
\bea
& & \frac d{dt} M_j(t)= \nonu\\
&  & =  -\frac 12 \int_\R \big[ - 3 \tilde u_x^2 - (c_{j-1}^0 + c_j^0) \tilde u^2 +
\frac 43 \tilde u^{3} -\frac 32 \beta\tilde u^4 \big]\psi_j' - \frac 12 \int_\R \tilde u^2 \psi_j^{(3)}  \label{dM1}\\
& &\qquad   + \frac 32 \int_\R \tilde u (1-\psi_j) [ (\varphi_c^2-c)\tilde u + (\sqrt{c}+ \varphi_c) \tilde u^2]_y \label{dM2}\\
& & \qquad +\frac{1}{2\sqrt{\beta}} x_N'(t)\int_\R \tilde u (1-\psi_j) \varphi_c'. \label{dM3}
\eea
Let us consider the term (\ref{dM1}). By definition of $\psi$, $|\psi^{(3)}| \leq \frac {\sigma_0}4 \psi'$, so that
\be\label{faux24}
\abs{\int_\R  \tilde u^2 \psi_j^{(3)}} \leq \frac 14 (c_{j-1}^0+ c_j^0) \int_\R \tilde u^2 \psi_j' .
\ee
In order to bound the term ${\displaystyle \int_\R (\frac 43 \tilde u^3 -\frac 32 \beta \tilde u^4) \psi_j'}$, one follows the argument of \cite{MMT}, splitting the real line in two different regions according to the position of each soliton, and the rest.  Following that argument, one finds
$$
\abs{\int_\R (\frac 43 \tilde u^3 -\frac 32 \beta \tilde u^4) \psi_j'} \leq K e^{- \sigma_0 (t + L)} + \frac 14 \sigma_0 \int_\R \tilde u^2 \psi_j'.
$$
Now we consider the term (\ref{dM2}). Note that one has 
$$
\abs{(1-\psi_j) (\varphi_c^2 -c)} \leq K e^{-\sigma_0 (t + L) },
$$
and a similar estimate is valid for the term $\abs{(1-\psi_j) (\sqrt{c}+\varphi_c)}$. Therefore 
$$
\abs{ \frac 32 \int_\R \tilde u (1-\psi_j) [ (\varphi_c^2-c)\tilde u + (\sqrt{c}+ \varphi_c) \tilde u^2]_y}\leq K e^{-\sigma_0(t + L)},
$$
Let us consider the term (\ref{dM3}). In this case, it is enough to recall that   
$$
\|(1-\psi_j) \varphi_c' \|_{L^2(\R)}\leq K e^{-\sigma_0 (t + L) },
$$
and $|x_N'(t)| \leq K\al$. Finally, we obtain for some $K>0$,
$$
\frac d{dt} M_j(t) \geq  -K e^{-\sigma_0 (t +L)} .
$$
Thus, by integrating between $0$ and $t$, we get the conclusion. Note that $K$ and $L$ are chosen independently of $t$.
\end{proof}

\begin{lem}[Quadratic control of the variation of $c_j(t)$]\label{PROOFL1}~

There exists $K>0$ independent of $K^*$, such that for all $t\in [0, T^*]$, 
\be\label{proof1}
\sum_{j=1}^{N-1} |c_j(t) - c_j(0) |\leq K( \|z(t)\|_{H^1(\R)}^2 + \| z(0) \|_{H^1(\R)}^2 + e^{-\sigma_0 L}).
\ee
\end{lem}

\begin{proof}
We proceed in several steps, following the proof given in \cite{MMT}.

\medskip

1. Note that from (\ref{EE}), and using a Taylor expansion of the function 
$$
f(M[Q_{c_j(t),\beta}] ):= \frac 23 c_j^{3/2}(t)
$$
around the point $s_0:= M[Q_{c_j(0),\beta}] $,\footnote{In particular, a simple computation using (\ref{dQc}) and $c=\frac1{9\beta}$ shows that 
$$
f'(s_0) = \frac{c^{1/2}}{\partial_c M[Q_{c,\beta}]} \Big|_{c=c_j(0)} = \frac 1{9c}(2c-c_j(0)) = \beta (2c-c_j(0)).
$$
} one has for some $K>0$,
\bee
& & \Big| \sum_{j=1}^{N-1}  \beta (2c-c_j(0)) ( M[Q_{c_j(t),\beta}]  - M[Q_{c_j(0),\beta}]  ) \Big|  \leq  \\
& & \quad  \leq K (\|z(t)\|_{H^1(\R)}^2 +\|z(0)\|_{H^1(\R)}^2) +Ke^{-\sigma_0 L}    + K \sum_{j=1}^{N-1}  (M[Q_{c_j(t),\beta}]  -M[Q_{c_j(0),\beta}] )^2.
\eee
Note in addition that for $\al_0$ small and $L_0>0$ large, from (\ref{dQc}),
\bee
 \abs{M[Q_{c_j(t),\beta}]  -M[Q_{c_j(0),\beta}]}  & =& \partial_c M[Q_{c,\beta}] \Big|_{c=c_j(0)} (c_j(t)-c_j(0)) + O(|c_j(t) -c_j(0)|^2)\\
& = & \frac{c_j^{1/2}(0)}{\beta(2c-c_j(0))} (c_j(t)-c_j(0)) + O(|c_j(t)-c_j(0)|^2).
\eee
Since $2c-c_j(0)\geq 2c-c_{N-1}(0) >\sigma_0>0$, one has
\bea\label{proof2}
& & \Big| \sum_{j=1}^{N-1}  \beta (2c-c_j(0)) ( M[Q_{c_j(t),\beta}]  - M[Q_{c_j(0),\beta}]  ) \Big|  \nonu\\
& & \qquad \leq  K ( \|z(t)\|_{H^1(\R)}^2 +\|z(0)\|_{H^1(\R)}^2 +e^{-\sigma_0 L} )  + K \sum_{j=1}^{N-1} |c_j(t) -c_j(0)|^2. \nonu\\
& & 
\eea
In other words, the left hand side above is of \emph{quadratic variation in $z$}. 

\medskip

2. Let 
\be\label{dj}
d_j(t) :=  \beta  \sum_{k=1}^{j} M[Q_{c_k(t),\beta}], \quad j=1, \ldots , N-1. \quad (=\hbox{the mass on the left})
\ee
We claim  that there exists $K>0$ such that, 
for all $j=1,\ldots, N-1$,
\be\label{proof3} 
|d_j(t)-d_j(0) |  \leq   (d_j(t)-d_j(0)) + K\|z(0)\|_{L^2(\R)}^2  + K\|z(t)\|_{L^2(\R)}^2 + K e^{-\sigma_0 L}.
\ee
Let us prove this last identity. Suppose $j\in \{ 1,\ldots, N-1\}$. First of all, if $d_j(t)-d_j(0)$ is nonnegative, there is nothing to prove. Let us assume that $d_j(t)-d_j(0)<0$, therefore we have to show that
$$
d_j(0)-d_j(t) \leq  K\|z(0)\|_{L^2(\R)}^2 +K\|z(t)\|_{L^2(\R)}^2 + K e^{-\sigma_0 L}.
$$  
Recall that from Lemma \ref{th:2}, one has $M_j(0)\leq M_j(t) + K e^{-\sigma_0 L}.$
On the other hand, from (\ref{Ortho2}), (\ref{R}), (\ref{tu}) and (\ref{Id5}),
\bee
M_j(t) & =& \frac 12 \int_\R (\tilde R^2 + 2 \tilde R \tilde z + \tilde z^2  )\tilde \psi_j(t) \\
& =& \frac 1{2} \sum_{k=1}^{N-1} \int_\R Q_{c_k,\beta}^2  \tilde \psi_j(t) +\frac 12 \int_\R \tilde z^2\tilde \psi_j(t) + O(e^{-\sigma_0 L})\\ 
& =& \frac 1{\beta}\Big[ \beta\sum_{k=1}^{j} M[Q_{c_k(t),\beta}] +  \frac 12\int_\R z^2\psi_j(t) \Big] + O(e^{-\sigma_0 L}).
\eee
Therefore,
\be\label{proof4}
d_j(t)-d_j(0) = -\frac 12 \int_\R (z^2\psi_j(t)-z^2\psi_j(0)) +  \beta (M_j(t) - M_j(0)) + O( e^{-\gamma_0 L}).
\ee
Using Lemma \ref{th:2}, (\ref{proof3}) follows easily. 

\medskip

\noindent
3. Conclusion. From the definition of $d_j(t)$ in (\ref{dj}), 
\bea 
&& \sum_{j=1}^{N-1}  \beta (2c-c_j(0)) (M[Q_{c_j(t),\beta}]  -M[Q_{c_j(0),\beta}] ) = \nonu\\ 
&& \quad =  (2c-c_{1}(0)) (d_{1}(t)-d_{1}(0)) + \sum_{j=2}^{N-1} (2c-c_j(0)) [ d_j(t)-d_{j-1}(t)
-(d_j(0)-d_{j-1}(0))]   \nonu \\
& &\quad =  (2c-c_{N-1}(0))(d_{N-1}(t) -d_{N-1}(0))  + \sum_{j=1}^{N-2}  (c_{j+1}(0)-c_{j}(0)) (d_j(t)-d_j(0)). \label{proof5} 
\eea
Therefore, by (\ref{proof2}) and (\ref{proof5}),
\be\label{unnom}
(\ref{proof5}) \leq  K \|z(t)\|_{H^1(\R)}^2 + K\|z(0)\|_{H^1(\R)}^2 + K e^{-\sigma_0 L} +  K\sum_{j=1}^{N-1} |c_j(t)-c_j(0)|^2.
\ee
Since $2c-c_{N-1}(0)\geq\sigma_0$ and $c_{j+1}(0)-c_{j}(0)\geq \sigma_0$, by (\ref{proof3}), one has
\bee
 \sigma_0 \sum_{j=1}^{N-1} | d_j(t) - d_j(0) | & \leq &   (2c-c_{N-1}(0)) |d_{N-1}(t)-d_{N-1}(0)| \\
 & & \qquad \qquad   +\sum_{j=1}^{N-2}  (c_{j+1}-c_j)(0) |d_j(t)-d_j(0)|\nonu \\
&  \leq &  (\ref{proof5})+ K\| z(0)\|_{L^2(\R)}^2 + K\| z(t)\|_{L^2(\R)}^2 + K e^{-\sigma_0 L}.
\eee
Thus, by (\ref{unnom}),
\be\label{djdj}
\sum_{j=1}^{N-1} |d_j(t) - d_j(0) |    \leq   K \|z(t)\|_{H^1(\R)}^2 + K \|z(0)\|_{H^1(\R)}^2+ Ke^{-\gamma_0 L} +  K \sum_{j=1}^{N-1} |c_j(t) -c_j(0)|^2.
\ee
Since for $j\geq 2$ one has 
\bee
|c_j(t)-c_j(0)| &\leq & K |M[Q_{c_j(t),\beta}]  -M[Q_{c_j(0),\beta}]| \\
& = & K| d_j(t) -d_{j-1}(t) - d_j(0) +d_{j-1}(0)|  \\
&\leq & K |d_j(t)-d_j(0)| + K |d_{j-1}(t)-d_{j-1}(0)|,
\eee
we obtain from (\ref{djdj})
$$
\sum_{j=1}^{N-1} |c_j(t) - c_j(0) |   \leq   K \|z(t)\|_{H^1(\R)}^2 + K\|z(0)\|_{H^1(\R)}^2 + Ke^{-\sigma_0 L}+ K \sum_{j=1}^{N-1} |c_j(t) -c_j(0)|^2.
$$
Choosing a smaller $\alpha_0$ and a larger $L_0$, depending on $K^*$, we can assume 
$K  |c_j(t) - c_j(0)| \leq 1/2$ and so
\be\label{proof6}
\sum_{j=1}^{N-1} |c_j(t) - c_j(0) | \leq K \|z(t)\|_{H^1(\R)}^2 + K \|z(0)\|_{H^1(\R)}^2 + Ke^{-\sigma_0 L}.
\ee
The proof is complete.
\end{proof}

\begin{lem}[Bootstrap]\label{proofl2}~

There exists $K>0$, independent of $K^*$, such that for all $t\in [0, T^*]$,
$$
\|z(t)|_{H^1(\R)}^2 \leq K ( \|z(0)\|^2_{H^1(\R)} +e^{-\sigma_0 L}).
$$
\end{lem}

\begin{proof}
By (\ref{EE}),
\bee
 \mathcal F(t) & = & \mathcal F(0) - \frac 23 \sum_{j=1}^{N-1} (c_j^{3/2}(t) -c_j^{3/2}(0))  + O(\|z(t)\|^3_{H^1(\R)}) + O(e^{-\sigma_0 L}) \\
&  = & - \sum_{j=1}^{N-1} \beta (2c-c_j(0))(M[Q_{c_j(t),\beta}]  -M[Q_{c_j(0),\beta}]) \\
& & \qquad + O\Big( \|z(0)\|_{H^1(\R)}^2 + \|z(t)\|^3_{H^1(\R)} + \sum_{j=1}^{N-1} |c_j(t) -c_j(0)|^2 + e^{-\sigma_0 L}\Big) \\
&  = & - (2c-c_{N-1}(0))(d_{N-1}(t) -d_{N-1}(0))  -  \sum_{j=1}^{N-2}  (c_{j+1}(0)-c_j(0)) (d_j(t)-d_j(0)) \\
& & \qquad + O\Big( \|z(0)\|_{H^1(\R)}^2 + \|z(t)\|^3_{H^1(\R)} + \sum_{j=1}^{N-1}|c_j(t) -c_j(0)|^2 + e^{-\sigma_0 L}\Big). 
\eee
On the other hand, note that from (\ref{proof2}) and Lemma \ref{th:2},
$$
d_j(t) -d_j(0) \geq -\frac 12 \int_\R (z^2\psi_j(t) -z^2\psi_j(0)) - Ke^{-\sigma_0 L}, \quad j=1,\ldots, N-1.
$$
Therefore
\be\label{dernier}
\mathcal{\tilde F}(t) \leq K \|z(0)\|_{H^1(\R)}^2 + K\|z(t)\|^3_{H^1(\R)} + K\sum_{j=1}^{N-1} (c_j(t) -c_j(0))^2 + e^{-\sigma_0 L},
\ee
with $\mathcal{\tilde F}(t)$ given by the formula
\bea\label{tFt}
& & \mathcal{\tilde F}(t)   := \nonu\\
& & =  \mathcal{F}(t) - \frac 12 \sum_{j=1}^{N-2} ((c_{j+1}(0)-2c) + (2c-c_j(0)))\int_\R z^2 \psi_j(t) \nonu \\
& & \qquad  - \frac 12 (2c-c_{N-1}(0))\int_\R z^2\psi_{N-1}(t)\nonu \\
& & = \mathcal F(t) - \frac 12 \sum_{j=2}^{N-1} (2c-c_{j}(0))\int_\R z^2 (\psi_j(t)-\psi_{j-1}(t))  - \frac 12 (2c-c_1(0))\int_\R z^2\psi_{1}(t) \nonu \\
&  & =  \frac 12\int_\R \big\{ z_x^2(t) + c(t,x) z^2(t) -3 (c-\varphi_c^2)z^2(t)   - 6 \sqrt{c} R z^2(t) + 3R^2 z^2(t)\big\},
\eea
with 
\bea\label{ctx}
c(t,x) & := & 2c - \sum_{j=2}^{N-1} (2c-c_{j}(0)) (\psi_j-\psi_{j-1})(t)  - (2c-c_1(0)) \psi_{1}(t) \nonu \\
& =& \!\!\! 2c\Big[ 1-   \sum_{j=2}^{N-1}  (\psi_j-\psi_{j-1})(t)  -\psi_{1}(t) \Big]  \nonu\\
& & \quad+  \sum_{j=2}^{N-1} c_{j}(0) (\psi_j-\psi_{j-1})(t)  +c_1(0) \psi_{1}(t).
\eea
We prove in Appendix \ref{B} that this quadratic form is coercive, in the sense that there exists $\la_0>0$ independent of $t$ and $K^*$ such that, thanks to (\ref{Ortho1}) and (\ref{Ortho2}),
\be\label{posit}
\mathcal{\tilde F}(t) \geq  \la_0 \| z(t) \|_{H^1(\R)}^2.
\ee
Therefore, from (\ref{dernier}), (\ref{proof1}), and taking $\al_0$ smaller if necessary, we obtain
$$
\|z(t) \|_{H^1(\R)}^2 \leq K \|z(0)\|_{H^1(\R)}^2 + K\|z(t)\|^3_{H^1(\R)}  + e^{-\sigma_0 L},
$$
and so
$$
\|z(t) \|_{H^1(\R)}^2 \leq K \|z(0)\|^2_{H^1(\R)}+ K e^{-\sigma_0 L},
$$
for some constant $K>0$, independent of $K^*$. Thus, the proof of  Lemma \ref{proofl2} is complete.
\end{proof}

We conclude the proof of Proposition \ref{Pr0}. From (\ref{smallness1}), Lemmas \ref{PROOFL1} and \ref{proofl2}, we have
\bee  
&& \Big\|u(t)  -\varphi_c (\cdot + ct + x_N(t)) +\sqrt{\beta} \sum_{j=1}^{N-1} Q_{c_j^0,\beta}(\cdot + \tilde c_j t +x_j(t))\Big\|_{H^1(\R)} \\ 
& & \quad \leq \|z(t)\|_{H^1(\R)} 
 + \sqrt{\beta}\Big\|\sum_{j=1}^{N-1} [Q_{c_j,\beta}(\cdot + \tilde c_j t +x_j(t)) -Q_{c_j^0,\beta}(\cdot + \tilde c_j t +x_j(t))] \Big\|_{H^1(\R)}\\
& &\quad \leq  \|z(t)\|_{H^1(\R)} + K \sum_{j=1}^{N-1} | c_j(t)- c_j^0| \\
& & \quad \leq  \| z(t) \|_{H^1(\R)} + K\sum_{j=1}^{N-1} |c_j(t)-c_j(0)|+ K\sum_{j=1}^{N-1} |c_j(0)-c_j^0|\\
&& \quad \leq  \| z(t) \|_{H^1(\R)} + K \|z(0)\|^2_{H^1(\R)} + K e^{-\sigma_0 L} + K \alpha  \quad  \leq  \;   K ( \alpha +e^{-\sigma_0 L} ),
\eee
where $K>0$ is a constant independent of $K^*$. Finally, choosing  $K^*=4 K$, we get the desired contradiction. The proof is complete.

\medskip

\bigskip

\section{Sketch of proof of Theorem \ref{EoMK}}\label{4}

\medskip

The proof follows the lines of \cite[Theorem 1]{Martel1} and the proof of Proposition \ref{Pr0} from the previous section. Let us assume the hypotheses of Theorem \ref{EoMK}. Let $T_n\to +\infty$ be an increasing sequence, and
$$
R(t) := \varphi_{c_N^0} (\cdot + c_N^0 t) + \sqrt{\beta} \sum_{j=1}^{N-1} Q_{c_j^0,\beta}(\cdot  + \tilde c_j t), \quad \tilde c_j = 3c_N^0-c_j^0.
$$ 
It is clear that $R(t)-\varphi_{c_N^0} (\cdot + c_N^0 t)$ is uniformly bounded in any $H^s(\R)$, $s\geq 0$. 

\medskip

Now we consider the following Cauchy problem
\bea\label{CP}
(u_{n})_t + ((u_n)_{xx} - u_n^3)_x =0,& &  \quad u_n(t,x)\in \R,\\
u_n(T_n) = R(T_n).& & 
\eea
From \cite{MV}, one has global existence of a unique solution $u_n(t)$ for (\ref{CP}), satisfying $u_n-\varphi_{c_N^0} \in C(\R,H^1(\R))$, with conserved energy (\ref{E}). The main part of the proof is to establish the following uniform estimates:

\begin{prop}[Uniform estimates]\label{Ue0}~

There exist $K, n_0>0$ such that for all $n\geq n_0$, and for all $t\in [T_{n_0}, T_n]$, one has
\be\label{Ue1}
\| u_n(t) - R(t) \|_{H^s(\R)} \leq K_s e^{-\sigma_0 t}, \quad s\in [0, 1^+],
\ee
with $\sigma_0>0$ defined in (\ref{sig0}).
\end{prop}

The proof of this result is similar to the proof of Proposition \ref{Pr0}, but it is easier since we do not need to modulate the scaling parameters $c_j(t)$ in Lemma \ref{dec}. In particular, the term $\sim \sum_{j}c_j^{3/2}(t)$ in (\ref{EE}) is constant. Lemma \ref{th:2} holds with no modifications. In order to control the directions $Q_{c_j,\beta}$  in (\ref{posit}), we only use Lemma \ref{th:2}, so Lemma \ref{PROOFL1} is not needed. The reader may consult \cite{Martel1} for a detailed proof.

\medskip

As a consequence of the above estimate, one has, up to a subsequence, and for all $t\geq T_{n_0}$,
\bee
u_n(t) -\varphi_{c_N^0}(\cdot +c_N^0 t) \rightharpoonup  u_0 \quad \hbox{ in } H^1(\R),\\
u_n(t) -\varphi_{c_N^0}(\cdot +c_N^0 t) \to  u_0 \quad \hbox{ in } L^2(K), 
\eee
for all $K\subset \R$ compact. On the other hand, note that from (\ref{CG2}) and (\ref{eqtu}) the function
$$
\tilde u_n(t,y) := \frac 1{\sqrt{\beta}} ( u_n(t, y - 3c_N^0 t ) -\varphi_{c_N^0} (y -2c_N^0t) ), \quad c_N^0= \frac 1{9\beta}, 
$$
satisfies the equation (\ref{eqtu}). Arguing as in \cite[eqn. (14)]{Martel1}, one has the following

\begin{lem}[Egorov estimate]\label{UEE}~

There exists $\ve_0>0$ such that, for all $0<\ve<\ve_0$, the following holds. There exists $R_0=R_0(\ve)>0$ such that for all $n\geq n_0$,
\be\label{UnifB}
\| u_n(T_{n_0}) -\varphi_{c_N^0} (\cdot + c_N^0T_{n_0} ) \|_{L^2(|x|>R_0)} \leq \ve.
\ee
\end{lem}

From (\ref{UnifB}) one has that 
$$
u_n(T_{n_0}) -\varphi_{c_N^0}(\cdot +c_N^0 T_{n_0}) \to  u_0 \quad \hbox{ in } L^2(\R), 
$$
and by interpolation and (\ref{Ue1}), the convergence is in $H^1(\R)$.  Let $U$ be the unique solution of (\ref{mKdV}) such that  $U(T_{n_0}) = \varphi_{c_N^0}(\cdot +c_N^0 T_{n_0})+ u_0$ (cf \cite[Proposition 3.1]{MV}). From the uniform $H^1(\R)$ continuity of the mKdV flow on compact sets of time, one has
$$
u_n(t) \to U(t) \quad \hbox{ in }ÊH^1(\R),
$$
for all $t\geq T_{n_0}$. Therefore, $\| u_n(t) - U(t)\|_{H^1(\R)} \to 0$ as $n\to +\infty$, for all $t\geq T_{n_0}$. Finally, passing to the limit
in (\ref{Ue1}), we get the desired existence conclusion.

\medskip

\noindent
{\bf Uniqueness.}  Using once again the transformation (\ref{CG2}), and the equation (\ref{eqtu}), we claim that from \cite{Martel1}, one has the following

\begin{lem}[Exponential decay]\label{Edec}~

Let $v\in \varphi_{c_N^0} (\cdot + c_N^0 t) + C(\R, H^1(\R))$ be a solution of (\ref{mKdV}) satisfying (\ref{lim4}). Then there exists $K, T_0>0$ such that
$$
\sup_{t\geq T_0} \| v(t) - \varphi_{c_N^0} (\cdot + c_N^0t + x_N^+)  - \sqrt{\beta}\sum_{j=1}^{N-1} Q_{c_j^0,\beta} (\cdot+ \tilde c_j t  +x_j^+) \|_{H^1(\R)} \leq Ke^{-\sigma_0 t}.
$$
\end{lem}
Using this property, the uniqueness result is just a consequence of the analysis carried out in \cite{Martel1}. We skip the details.

\bigskip

\appendix

\section{Proof of (\ref{posit})}\label{B}

In this section we sketch the proof of (\ref{posit}). See e.g. \cite{MMT} for a detailed, similar proof.  First of all, note that from (\ref{tFt}), (\ref{R}) and (\ref{tRtz}) one has
$$
\mathcal{\tilde F}(t) = \frac 12 \int_\R (z_x^2 + c(t,x) z^2 -3(c-\varphi_c^2) z^2 - (2\tilde R -3\beta \tilde R^2)z^2 ), 
$$
with $c(t,x)$ given in (\ref{ctx}).

\medskip

\noindent
1. We recall the following well-known result.

\begin{lem}[Positivity of the Zhidkov functional, see \cite{Z,MV}]\label{Zi}~

There exists $\la_0>0$ such that for all $z\in H^1(\R)$, with  $\int_\R z \varphi_c' =0,$
one has
\be\label{443}
\frac 12 \int_\R (z_x^2 + 2c z^2 - 3(c-\varphi_c^2)z^2 ) \geq \la_0 \int_\R (z_x^2 + z^2).
\ee
\end{lem}

\medskip

\noindent
2. Let  $\Phi\in C^2(\R)$, with $\Phi(s)=\Phi(-s)$, $\Phi'\leq 0$ on $\R^+$, 
and
$$
\Phi(s)=1 \hbox{ on $[0,1]$;}\quad 
\Phi(s)=e^{-s} \hbox{ on $[2,+\infty)$,}\quad
e^{-s}\le \Phi(s)\le 3 e^{-s} \quad \hbox{on
$\R^+$.}
$$ 
Finally, let $\Phi_{B}(s) :=\Phi (\frac sB).$

\begin{lem}[Localized coercivity, see e.g. \cite{MMT,We1}]\label{loc_pos}~

There exists $B_0, \la_0>0$ such that, for all $B>B_0$, if $z\in H^1(\R)$ satisfies
$$
\int_\R  Q_c z =\int_\R Q_c' z=0,
$$
then
\be\label{singlelocal}
\int_\R \Phi_{B} (z_x^2- (2Q_c^{2} - 3\beta Q_c^2 )z^2 + cz^2) \geq \la_0 \int_\R \Phi_{B} (z_x^2+z^2).
\ee
\end{lem}
Note that a similar argument can be followed in order to prove a localization property for the Zhidkov functional considered in Lemma \ref{Zi}. We skip the details.

\medskip 

\noindent
3. Now we perform a localization argument, as in \cite{MMT}. One has from (\ref{443}),
$\mathcal{\tilde F}(t)  =   \sum_{j=1}^{N} \mathcal{F}_j(t) + (\mathcal{\tilde F}(t) -  \sum_{j=1}^{N} \mathcal{F}_j(t)),$ with 
$$
\mathcal F_j(t) := \frac 12 \int_\R \Phi_{B,j}(z_x^2 + c_j(t) z^2 -(2Q_{c_j,\beta} -3\beta Q_{c_j,\beta}^2 )z^2), \; \Phi_{B,j} :=\Phi_B(x + \tilde c_j t + x_j(t)),
$$
for all $j=1,\ldots, N-1$, and
$$
\mathcal F_N(t) := \frac 12 \int_\R \Phi_{B,N}(z_x^2 + 2c z^2 - 3(c-\varphi_c^2)z^2), \quad \Phi_{B,N}:=\Phi_B(x + c t + x_N(t)).
$$
From Lemma \ref{loc_pos}, one has for $B$ large enough,
$$
\mathcal{F}_j(t) \geq \la_0 \int_\R \Phi_{B,j}(z_x^2 +  z^2)(t,x)dx,
$$
for all $j=1,\ldots, N$. On the other hand,
\bee
\mathcal{\tilde F}(t) -\sum_{j=1}^N \mathcal{F}_j(t) & =&  \int_\R \Big( 1 - \sum_{j=1}^N \Phi_{B,j} \Big) (z_x^2 + c(t,x) z^2) \\
& & + \sum_{j=1}^{N-1} \int_\R \Phi_{B,j} (c(t,x) -c_j(t))z^2 +  \int_\R \Phi_{B,N} (c(t,x) - 2c) z^2 \\
& &  - \frac 12\int_\R \Big( 2 \tilde R -3\beta \tilde R^2 -\sum_{j=1}^{N-1} \Phi_{B,j}(2Q_{c_j,\beta} -3\beta Q_{c_j,\beta}^2) \Big)z^2 \\
& & - \frac 32 \int_\R (1 -\Phi_{B,N})(c-\varphi_c^2)z^2 
\eee
Each term above can be treated following the lines of the proof of Lemma 4 in \cite{MMT}, and it is proved that for all $B$ large, the above terms can be estimated by
$$
\geq -\frac 1{16} \la_0 \int_\R (z_x^2+z^2).
$$
The final conclusion is that for $B$ large enough, but independent of $z$,
$$
\mathcal{\tilde F}(t)  \geq  \frac 1{8}\la_0 \sum_{j=1}^N \int_\R \Phi_{B,j}(z_x^2 + z^2 ) \geq  \frac 1{16}\la_0 \int_\R (z_x^2+z^2),
$$
for some $\la_0>0$ independent of $z(t)$ and $B$. Thus  the proof of (\ref{posit}) is complete.

\section{Proof of some identities}

\begin{lem}[Identities]\label{Id}~

Let $Q_{c,\beta}$ be the Gardner soliton from (\ref{SolG}). Then one has
\ben
\item \emph{Basic identities}.
\be\label{Id1}
Q_{c,\beta}'' = cQ_{c,\beta} -Q_{c,\beta}^2 +\beta Q_{c,\beta}^3, \qquad Q_{c,\beta}'^2 = cQ_{c,\beta}^2 -\frac 23 Q_{c,\beta}^3 +\frac \beta 2Q_{c,\beta}^4.
\ee
\item \emph{Integrals}.
\be\label{Id3}
\beta \int_\R Q_{c,\beta}^3  = -c \int_\R Q_{c,\beta} + \int_\R Q_{c,\beta}^2, \qquad \beta \int_\R Q_{c,\beta}^4  = - \frac 43c \int_\R Q_{c,\beta}^2 + \frac{10}{9}\int_\R Q_{c,\beta}^3,
\ee
and
\be\label{Id5}
\int_\R Q_{c,\beta} = \frac 32 \beta \int_\R Q_{c,\beta}^2  + 6\sqrt{c}. 
\ee
\item \emph{Energy}. Let 
$$
E_\beta[Q_{c,\beta}] := \frac 12 \int_\R Q_{c,\beta}'^2 -\frac 13 \int_\R Q_{c,\beta}^3 +\frac \beta 4  \int_\R Q_{c,\beta}^4.
$$
Then one has
\be\label{EQb}
E_\beta[ Q_{c,\beta}] =\frac 2{3\beta} c^{3/2} -\frac{1}{9\beta}\int_\R Q_{c,\beta}^2.
\ee
\item \emph{Weinstein's condition}. For $c<\frac 2{9\beta}$,
\be\label{dQc}
\partial_c \frac 12 \int_\R Q_{c,\beta}^2 = \frac{9c^{1/2}}{2-9\beta c}.
\ee
\een
\end{lem}

\begin{proof}
The first identity in (\ref{Id1}) is just the elliptic equation for $Q_{c,\beta}$, obtained by replacing in (\ref{Ga}). The  second one follows from the first identity in (\ref{Id1}), after multiplication by $Q_{c,\beta}'$ and integration in space.  

On the other hand, the first identity in (\ref{Id3}) follows after integration of (\ref{Id1}). In the same form, the second identity in (\ref{Id3}) is a consequence of the first identity in (\ref{Id1}) and the integration of the second one in (\ref{Id1}) against $Q_{c,\beta}$.

Let us prove (\ref{Id5}). Note that from (\ref{Id1}),
$$
\big(\frac{Q_{c,\beta}'}{Q_{c,\beta}} \big)' =  -\frac 13 Q_{c,\beta}+\frac \beta2 Q_{c,\beta}^2.
$$
Using the definition of $Q_{c,\beta}$ from (\ref{SolG}), and integrating, one gets
$$
-2\sqrt{c}  =  -\frac 13 \int_\R Q_{c,\beta}+\frac \beta2 \int_\R Q_{c,\beta}^2,
$$
namely (\ref{Id5}). Now we prove  (\ref{EQb}).  From (\ref{Id1}) and (\ref{Id3})
\bee
E_\beta[ Q_{c,\beta}]  & =& \frac 12 \int_\R Q_{c,\beta}'^2 -\frac 13 \int_\R Q_{c,\beta}^3 +\frac\beta 4\int_\R Q_{c,\beta}^4  \\
& =&  \frac 12c \int_\R Q_{c,\beta}^2 -\frac 23 \int_\R Q_{c,\beta}^3 +\frac\beta 2\int_\R Q_{c,\beta}^4 \\
& =& -\frac 16 c\int_\R Q_{c,\beta}^2 -\frac 19\int_\R Q_{c,\beta}^3 = \frac c{9\beta}\int_\R Q_{c,\beta}  -(\frac c6 +\frac{1}{9\beta})\int_\R Q_{c,\beta}^2. 
\eee
Using (\ref{Id5}), we obtain (\ref{EQb}), as desired.

\medskip

Finally, let us prove (\ref{dQc}). From the definition (\ref{SolG}), one has
$$
\partial_c Q_{c,\beta}(s) = \frac 1c\Big[(1+ \frac {9\beta c}{4\rho^2})Q_{c,\beta} -\frac{3\beta}{4\rho^2}Q_{c,\beta}^2+ \frac 12s Q_{c,\beta}'(s)\Big]; 
$$
therefore, using (\ref{Id3}) and (\ref{Id5}), 
\bee
\partial_c\frac 12 \int_\R Q_{c,\beta}^2  & = & \int_\R Q_{c,\beta} \partial_c Q_{c,\beta} =  \frac 3{4c} \int_\R Q_{c,\beta}\Big[(1 +  \frac {3\beta c}{\rho^2} )Q_{c,\beta} -\frac{\beta }{\rho^2}Q_{c,\beta}^2\Big] \\
& =& \frac 3{4c} \Big[ (1 +  \frac {3\beta c}{\rho^2}  -\frac 1{\rho^2}) \int_\R Q_{c,\beta}^2    + \frac{c}{\rho^2}\int_\R Q_{c,\beta}  \Big] \\
& =& \frac 3{4c} \Big[ (1 +  \frac {3\beta c}{\rho^2}  -\frac 1{\rho^2} +\frac {3\beta c}{2\rho^2}) \int_\R Q_{c,\beta}^2    +  \frac{6c^{3/2}}{\rho^2} \Big]  = \frac{9c^{1/2}}{2\rho^2}.
\eee
\end{proof}

\medskip

\noindent
{\bf Acknowdlegments.} I would like to thank Yvan Martel, Frank Merle,  Miguel Angel Alejo, Luis Vega and Manuel del Pino for several remarks and comments on a first version of this paper. This work was in part written at the University of the Basque Country, and the University of Versailles. The author has been partially funded by grants Anillo ACT 125 CAPDE and Fondo Basal CMM.

\medskip

\end{document}